\documentclass[final,onefignum,onetabnum]{siamart220329}

\usepackage{braket,amsfonts}

\usepackage{array}

\usepackage[caption=false]{subfig}
\usepackage{pgfplots}

\newsiamthm{claim}{Claim}
\newsiamremark{remark}{Remark}
\newsiamremark{hypothesis}{Hypothesis}
\crefname{hypothesis}{Hypothesis}{Hypotheses}

\usepackage{algorithmic}

\usepackage{graphicx,epstopdf}

\Crefname{ALC@unique}{Line}{Lines}

\usepackage{amsopn}

\usepackage{xspace}
\usepackage{bold-extra}
\usepackage[most]{tcolorbox}

\colorlet{texcscolor}{blue!50!black}
\colorlet{texemcolor}{red!70!black}
\colorlet{texpreamble}{red!70!black}
\colorlet{codebackground}{black!25!white!25}

\lstdefinestyle{siamlatex}{%
  style=tcblatex,
  texcsstyle=*\color{texcscolor},
  texcsstyle=[2]\color{texemcolor},
  keywordstyle=[2]\color{texemcolor},
  moretexcs={cref,Cref,maketitle,mathcal,text,headers,email,url},
}

\tcbset{%
  colframe=black!75!white!75,
  coltitle=white,
  colback=codebackground, 
  colbacklower=white, 
  fonttitle=\bfseries,
  arc=0pt,outer arc=0pt,
  top=1pt,bottom=1pt,left=1mm,right=1mm,middle=1mm,boxsep=1mm,
  leftrule=0.3mm,rightrule=0.3mm,toprule=0.3mm,bottomrule=0.3mm,
  listing options={style=siamlatex}
}

\newtcblisting[use counter=example]{example}[2][]{%
  title={Example~\thetcbcounter: #2},#1}

\newtcbinputlisting[use counter=example]{\examplefile}[3][]{%
  title={Example~\thetcbcounter: #2},listing file={#3},#1}

\DeclareTotalTCBox{\code}{ v O{} }
{ 
  fontupper=\ttfamily\color{black},
  nobeforeafter,
  tcbox raise base,
  colback=codebackground,colframe=white,
  top=0pt,bottom=0pt,left=0mm,right=0mm,
  leftrule=0pt,rightrule=0pt,toprule=0mm,bottomrule=0mm,
  boxsep=0.5mm,
  #2}{#1}

\patchcmd\newpage{\vfil}{}{}{}
\usepackage{multirow}
\usepackage{enumerate}

\usepackage{amsmath, amsfonts, amssymb,mathrsfs}
\usepackage{bm}
\usepackage{listings}

\hypersetup{
	bookmarks=true,         
	unicode=false,          
	pdftoolbar=true,        
	pdfmenubar=true,        
	pdffitwindow=true,      
	pdftitle={},    
	pdfauthor={Junyuan Lin},     
	pdfsubject={},   
	pdfnewwindow=true,      
	pdfkeywords={}, 
	colorlinks=true,       
	linkcolor=red,          
	citecolor=blue,        
	filecolor=magenta,      
	urlcolor=cyan           
}

\renewcommand{\vec}[1]{\mathbf{#1}}

\everymath{\displaystyle}

\begin{tcbverbatimwrite}{tmp_\jobname_header.tex}
\title{Solving Graph Laplacians via Multilevel Sparsifiers\thanks{Submitted to the editors for the 2022 Copper Mountain Special Issue
of SIAM Journal on Scientific Computing.}}

\author{Xiaozhe Hu\thanks{Department of Mathematics, Tufts University, Medford, MA (\email{xiaozhe.hu@tufts.edu}).}
\and Junyuan Lin\thanks{Department of Mathematics, Loyola Marymount University, Los Angeles, CA (\email{junyuan.lin@lmu.edu}).}}

\headers{Solving Graph Laplacians via Multilevel Sparsifiers}{Xiaozhe Hu and Junyuan Lin}
\end{tcbverbatimwrite}
\input{tmp_\jobname_header.tex}

\ifpdf
\hypersetup{ pdftitle={Solving Graph Laplacians via Multilevel Sparsifiers} }
\fi

\begin{document}
\maketitle

\begin{tcbverbatimwrite}{tmp_\jobname_abstract.tex}
\begin{abstract}
We consider effective preconditioners for solving Laplacians of general weighted graphs. Theoretically, spectral sparsifiers (SSs) provide preconditioners of optimal computational complexity. However, they are not easy to use for real-world applications due to the implementation complications. Multigrid (MG) methods, on the contrary, are computationally efficient but lack of theoretical justifications. To bridge the gap between theory and practice, we adopt ideas of MG and SS methods and proposed preconditioners that can be used in practice with theoretical guarantees. We expand the original graph based on a multilevel structure to obtain an equivalent expanded graph. Although the expanded graph has a low diameter, a favorable property for constructing SSs, it has negatively weighted edges, which is an unfavorable property for the SSs. We design an algorithm to properly eliminate the negatively weighted edges and prove that the resulting expanded graph with positively weighted edges is spectrally equivalent to the expanded graph, thus, the original graph. Due to the low-diameter property of the positively-weighted expanded graph preconditioner (PEGP), existing algorithms for finding SSs can be easily applied. To demonstrate the advantage of working with the PEGP, we propose a type of SS, multilevel sparsifier preconditioner (MSP), that can be constructed in an easy and deterministic manner. We provide some preliminary numerical experiments to verify our theoretical findings and illustrate the practical effectiveness of PEGP and MSP in real-world applications.
\end{abstract}


\end{tcbverbatimwrite}
\input{tmp_\jobname_abstract.tex}

\section{Introduction}
Graph Laplacians naturally arise in large-scale computations of various applications. For example, solving systems with weighted graph Laplacians is the core component for solving ranking and user recommendation problems~\cite{hirani2010least,jiang2011statistical,HodgeRank}. In~\cite{cao2014new,DSD:2013,DSD2018,cowen2021random}, similarities of proteins are calculated by solving graph Laplacian systems associated with the protein interaction networks. Furthermore, these similarities are used in clustering and labeling proteins. In addition, the marriage of graph Laplacian with topics such as Convolutional Neural Networks and tensor decomposition has also been a dominating trend~\cite{he2006tensor,Bronstein.M;Bruna.J;LeCun.Y;Szlam.A;Vandergheynst.P2017a,Kipf.T;Welling.M2016a,Henaff.M;Bruna.J;LeCun.Y2015a,Bruna.J;Zaremba.W;Szlam.A;LeCun.Y2013a,shaham2018spectralnet,li2018adaptive}. Advanced algorithms that adapt graph Laplacian properties to improve performance are key components in modern methods for image reconstruction, clustering image data-sets, and classification~\cite{agarwal2006higher,narita2012tensor,GLTD,steerableGL}. 
	
Two major approaches to solve a linear system of equations are direct and iterative (indirect) methods. Direct methods are variants of Gaussian Elimination which might become expensive when the graph size gets large. In this case, it is more suitable to consider iterative methods that provide successively better approximations to the solutions. One example is the Conjugate Gradient (CG) method. However, the convergence of the iterative usually depends on the condition number of the linear system. A graph Laplacian matrix, $\bm{L}$, is symmetric positive semidefinite (has $\operatorname{span}\{ \vec{1} \}$ as its nullspace), and its condition number $\kappa(\bm{L})$ is defined as the ratio between the largest eigenvalue of $\bm{L}$ and the smallest non-zero eigenvalue of $\bm{L}$. In practice, $\kappa(\bm{L})$ could be huge, especially for large-scale graphs, therefore it is difficult to solve the graph Laplacian systems $\bm{L}\vec{x} = \vec{b}$ using standard iterative methods, such as CG.  

	
To accelerate the iterative methods, one can use preconditioning. A good preconditioner for a matrix $\bm{L}$ is a matrix $\bm{R}$ that makes solving the following systems of linear equations in $\bm{RL}$ easy.
\begin{equation}\label{eqn:graphL_precsys}
	\bm{R}\bm{L}\vec{x}=\bm{R}\vec{b},
\end{equation}
where $\bm{R}\in\mathbb{R}^{n\times n}$ is noted as the preconditioner of $\bm{L}$. (Note that we use the notation $\bm{R}$ for preconditioner instead of the more frequently used $\bm{R}^{-1}$, since we do not assume $\bm{R}$ is invertible.)
	
More specifically, we want to choose an $\bm{R}$ that approximates the Moore-Penrose pseudoinverse, $\bm{L}^{\dagger}$, of $\bm{L}$. When $\bm{R} \approx \bm{L}^{\dagger}$, we have $\kappa(\bm{R}\bm{L})\approx\kappa(\bm{L}^{\dagger}\bm{L})\approx\kappa(\bm{I})\ll\kappa(\bm{L})$. It is clear that finding suitable $\bm{R}$'s that are also inexpensive to compute is the key challenge for designing efficient preconditioned iterative methods. Based on the support theory~\cite{Vaidya90}, Gremban et al. introduced the support-tree preconditioners in 1995~\cite{alon1995graph}, and this idea of using spanning trees to approximate the original graph became popular in the computer science field. For the last decades, researchers have developed preconditioners using spanning trees and sparsifiers~\cite{bern2001support, boman2001spanning, boman2003support, spielman2004nearly, elkin2008lower,abraham2008nearly,koutis2011nearly,abraham2012using,Kelner:2013}. The best-performing support-tree preconditioners out of these are sparsifiers built from low-stretch spanning trees that can achieve nearly $O(m\log n)$ time for strictly diagonally dominant linear systems. Despite the theoretical success of those preconditioners, their implementations are still behind because finding low-stretch spanning trees for general graphs is complicated. Those spectral equivalent sparsifiers are usually constructed by randomly adding edges to the low-stretch spanning tree. Therefore, without an efficient implementation of the low-stretch spanning tree algorithm, it remains difficult to use those preconditioners for practical applications. 
	
As another important family of preconditioners, the algebraic multigrid (AMG) algorithms are frequently applied to solve linear graph Laplacian systems in practice. The standard AMG method was proposed to solve partial differential equations (PDEs) and mainly involves two parts: smoothing the high-frequency errors on the fine levels and eliminating the low-frequency errors on the coarse grids~\cite{xu2017algebraic,Vassilevski.P2008,Ruge.J;Stuben.K1987,AMG_1984,ruge1984efficient,ruge1983algebraic}. AMG is proven to be one of the most successful iterative methods in practical applications and many AMG methods have been developed to solving graph Laplacian systems, such as combinatorial multigrid~\cite{KOUTIS}, Lean AMG~\cite{livne2012lean}, Algebraic multilevel preconditioners based on matchings/aggregations~\cite{Kim.H;Xu.J;Zikatanov.L2003,notay2010aggregation,Brannick.J;Chen.Y;Kraus.J;Zikatanov.L.2013a}, and AMG with adaptive aggregation schemes~\cite{d2013adaptive,hu2019adaptive}. While multilevel methods work robustly, are deterministic, and are widely used in practice, they lack theoretical guarantees for general graphs.
	
In this paper, we propose new preconditioners that combine the pros of support-tree and AMG preconditioners while avoiding their cons. It is noted by researcher that finding a low-stretch spanning tree on low diameter graphs is relatively easy~\cite{alon1995graph,elkin2008lower,koutis2014approaching}. Inspired on this point, we aim to construct a spectrally equivalent graph to any given general graph, which has lower diameter so that finding low-stretch trees/sparsifiers is feasible in practice. MG methods become our motivation for such construction. As described in~\cite{griebel1994multilevel}, MG method can be viewed as a standard iterative preconditioner for solving a linear system with multilevel expanded structure. We adopt this idea and construct a multilevel expanded graph from the original graph. More precisely, we take the following two steps to build a low diameter, positive-weighted expanded graph, which is easier to construct and feasible for theoretical analysis using existing support tree theory, compared to the original system:
		\begin{itemize}
			\item Step 1: We expand the original graph $G$ using AMG construction to the expanded graph $\widetilde{G}^{\mu}$. The multilevel structure of expanded graph $\widetilde{G}^{\mu}$ guarantees that it is of low diameter.
			\item Step 2: Since Step 1 introduces negative edges into the expanded graph, we extract the only positive subgraph $\widetilde{H}^{\mu}$ out of $\widetilde{G}^{\mu}$
	\end{itemize}
The resulting positive-weighted expanded graph preconditioner (PEGP) $\widetilde{H}^{\mu}$ has only positive weights, small diameter, and multilevel structure, which makes applying any spanning trees/sparsifiers finding algorithms easy. We define those spanning sparsifiers as multilevel sparsifier preconditioner (MSP) and include one deterministic method to extract a tree-structured, low-stretch sparsifier on $\widetilde{H}^{\mu}$ as an illustrative example.

The paper is organized as follows: section~\ref{sec:Pre} includes preliminaries and notations for graph and aggregation. Using these building blocks, we first extend the original graph $G$ into a multilevel expanded graph $\widetilde{G}^{\mu}$ with the original graph on its fine level in Section~\ref{sec:G_tilde}. Due to the multilevel structure of $\widetilde{G}^{\mu}$, the diameter of the original graph $G$ shrinks. In Section~\ref{sec:H_mu}, we further find the PEGP $\widetilde{H}^{\mu}$ of the expanded graph $\widetilde{G}^{\mu}$ that only contain positive edges. We prove that the graph Laplacian $\bm{L}_{\widetilde{H}^{\mu}}$ is spectrally equivalent to $\bm{L}_{\widetilde{G}^{\mu}}$. It is much easier to find spanning trees/subgraphs on $\widetilde{H}^{\mu}$ than on $G$, so we introduced a spanning tree/sparsifier MSP in Section~\ref{sec:MSP}. Numerical results are included in Section~\ref{sec:num} to test the robustness of the proposed preconditioners.
	
\section{Preliminaries}\label{sec:Pre}
This section includes notations and preliminaries of graph and aggregation that are useful for defining the expanded graph, the PEGP, and the MSP.
	
\subsection{Graph}\label{sec:graph}
Consider a simply connected graph $G = (E, V, \omega)$, where $E$ is the edge set, $V$ is the node set, and $\omega$ is the edge weight set. Here, $|V| = n$ and $|E| = m$ denote the number of nodes and edges in graph $G$, respectively. $E^+(G)$ and $E^-(G)$ denote positive and negative edge set in $G$, respectively. 

$\bm{L}_G$ represents the graph Laplacian that corresponds to $G$. Note that, $\bm{L}_G = \bm{B}^T\:\bm{W}\bm{B}$. Here, $\bm{B}^T\in \mathbb{R}^{n\times m}$ is the incidence matrix and $\bm{W}\in \mathbb{R}^{m\times m}$ is the diagonal weight matrix of graph $G$. When $\omega$ only contains positive elements, we are able to write $\bm{W}=\bm{W_{1/2}}^T\: \bm{W_{1/2}}$, where $\bm{W_{1/2}}\in \mathbb{R}^{m\times m}$ is the diagonal matrix with square-root of the edge weights.

Using graph $G$, in this work, we build the expanded graph $\widetilde{G}^{\mu} = (\tilde{E}, \tilde{V}, \tilde{\omega}(\mu))$ based on a aggregation-based multigrid structure similar to~\cite{gremban1996combinatorial}. $\tilde{n}$ is the number of nodes in the expanded graph $\widetilde{G}^{\mu}$ and $\mu$ is a weight parameter we set to control the condition number of the preconditioned graph Laplacian system. The detailed construction of the expanded graph is illustrated in Section~\ref{sec:G_tilde}. 

\subsection{Aggregation}\label{sec:agg}
We consider the original graph $G$ and carefully choose nodes in $G$ into different aggregations. The aggregates now become nodes in the next level just as the construction of AMG. We then continue to aggregate nodes until the desired level is achieved. 
	
One can observe that there can be infinitely many ways to aggregate nodes on each level, such as aggregations based on heavy edge coarsening~\cite{HEM,KARYPIS199896} and maximal independent set~\cite{young1974note}, and these methods can all be applied to define multilevel structure. As suggested in the appendix of~\cite{HodgeRank}, aggregating two nodes that are connected by a heavy edge is recommended because it minimizes the condition number of $\bm{L}_{\widetilde{G}^{\mu}}$ in a two-level setting. Follow this idea, we use the maximal weighted matching (MWM)~\cite{galil1986ev} in the numerical implementations in Section~\ref{sec:num}. MWM visits the vertices in the graph in random order and matches each unaggregated vertex to its neighbor with the maximal edge weight. The following Figure~\ref{fig:MWM} shows an example of MWM algorithm implemented on a graph with anisotropic diffusion.  
	\begin{figure}[h!]
		\begin{center}
			\includegraphics[width=10cm]{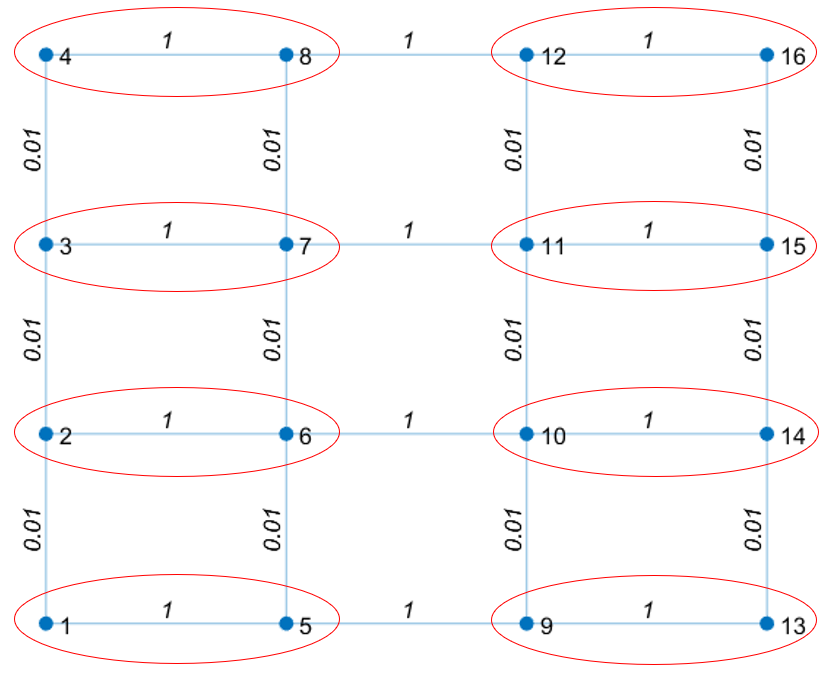}
		\end{center}
		\caption{The maximal weighted matching on an example graph}\label{fig:MWM}
	\end{figure} 

In Figure~\ref{fig:MWM}, the red circles represent nodes in the same aggregation. We can observe that all the aggregations contain two nodes linked by the heaviest incident edge. If there are leftover nodes after performing the matching scheme, we add each leftover node to the aggregation that has its heaviest neighbor.

Based on these aggregates, We are able to define the prolongation operator. Denote $n_{\ell}$ as the number of nodes on the $\ell^{th}$ level graph and the prolongation operator $\bm{P}^{\ell}_{\ell-1}\in \mathbb{R}^{n_{\ell-1}\times n_{\ell}}$ from level $\ell-1$ to level $\ell$ as follow,
	\begin{equation}
		\bm{P}^{\ell}_{\ell-1} (i, j) = \begin{cases} 1, & \text{if node $i$ is in aggregation $j$}, \\
			0, & \text{otherwise}.
		\end{cases}
	\end{equation}
	Furthermore, the prolongation operator $\bm{P}^{\ell}\in \mathbb{R}^{n\times n_{\ell}}$ from level $1$ to level $\ell$ is defined as follows,
	\begin{equation}\label{def:P}
		\bm{P}^{\ell}= \bm{P}^{2}_{1}\cdot \bm{P}^{3}_{2}\cdots\bm{P}^{\ell}_{\ell-1}.
	\end{equation}
	
	With all these components, we can define a composite prolongation operator which incorporates a weight parameter $\mu$ as follows,
	\begin{equation}
	    \{\bm{P}(\mu)\}^{\ell}: =
		\begin{bmatrix}
			 \bm{I} & (-\mu\bm{P^2}) & ((-\mu)^2\bm{P^3}) & \cdots & ((-\mu)^{\ell-1}\bm{P^{\ell}}),
		\end{bmatrix}
	\end{equation}
	where $\{\bm{P}(\mu)\}^{\ell}\in\mathbb{R}^{n\times \tilde{n}}$. This composite prolongation operator $\{\bm{P}(\mu)\}^{\ell}$ is an essential element in defining $\bm{L}_{\widetilde{G}^{\mu}}$ as shown in the next section.
	
	\section{Building the Positively Weighted Expanded Graph}
	In this section, we first explain in detail how to construct the multilevel expanded graph $\widetilde{G}^{\mu}$ so that its diameter is relative small but spectrally equivalent to the original graph. In particular, in Section~\ref{sec:G_tilde}, we expand the graph from $G$ to $\widetilde{G}^{\mu}$ and justify the legitimacy by showing the equivalence between the linear systems of graph Lapcians~$\bm{L}_G$ and $\bm{L}_{\widetilde{G}^{\mu}}$. Due to the existence of negative edges in $\widetilde{G}^{\mu}$, we further extract the PEGP $\widetilde{H}^{\mu}$ from $\widetilde{G}^{\mu}$. In Section~\ref{sec:H_mu}, we perform a spectral analysis between the two and show their spectral equivalence when parameter $\mu$ is chosen properly. Due the multilevel structure, the PEGP $\widetilde{H}^{\mu}$ has relatively small diameter which is favorable for finding a good spectral sparsifiers.  To demonstrate that, in Section~\ref{sec:MSP}, we propose a sparsifier, MSP, of $\widetilde{H}^{\mu}$ that has tree-like structure. 
	 
	\subsection{Expanding the Graph}\label{sec:G_tilde}
	Consider a simply connected graph $G$ and the linear system 
	\begin{equation}\label{eqn:originalsys}
		\bm{L}_G \vec{x}= \vec{b}.
	\end{equation}
	Since $\bm{L}_G$ is semi-definite and might be ill-conditioned, it is challenging to efficiently solve the linear system either by direct or simple iterative methods. Based on the support tree construction introduced by Gremban~\cite{gremban1996combinatorial} and the multilevel structure from the MG method~\cite{griebel1994multilevel}, we construct the expanded graph $\widetilde{G}^{\mu}$ such that solving its graph Laplacian~$\bm{L}_{\widetilde{G}^{\mu}}$,
	\begin{equation}\label{eqn:expandedsys}
		\bm{L}_{\widetilde{G}^{\mu}} \vec{\tilde{x}}= \vec{\tilde{b}}, 
	\end{equation}
	is mathematically equivalent to solving~\eqref{eqn:originalsys}.
	
	We build the graph Laplacian~$\bm{L}_{\widetilde{G}^{\mu}}$ of the expanded graph~$\widetilde{G}^{\mu}$ by the following formula: 
		\begin{align}
			\bm{L}_{\widetilde{G}^{\mu}} &= 
			\begin{bmatrix}
				\bm{I}\\
				(-\mu\bm{P^2})^T \\
				((-\mu)^2\bm{P^3})^T \\
				\vdots \\
				((-\mu)^{\ell-1}\bm{P^{\ell}})^T
			\end{bmatrix}\: \bm{L}_G\: 
			\begin{bmatrix}
				\bm{I} & (-\mu\bm{P^2}) & ((-\mu)^2\bm{P^3}) & \cdots & ((-\mu)^{\ell-1}\bm{P^{\ell}})
			\end{bmatrix}
		 \nonumber	\\
			&= (\{\bm{P}(\mu)\}^{\ell})^T \: \bm{L}_G\: \{\bm{P}(\mu)\}^{\ell} \label{def:L_ell}
		\end{align}
	Here, $\bm{I}\in \mathbb{R}^{n\times n}$ is the identity matrix and $0<\mu< 1$ is a scaling parameter that controls the conditional number of the preconditioner. We discuss the choice of $\mu$ in detail in Section~\ref{sec:H_mu}. Such a construction has been proposed in~\cite{griebel1994multilevel} for recasting the MG algorithm as an iterative method for the semidefinite expanded linear system.  Note that since $\{\bm{P}(\mu)\}^{\ell} \cdot \bm{1}^{\tilde{n}} = c\cdot\bm{1}^n$, where $c$ is a constant and $\bm{L}_G \cdot \bm{1}^n=\bm{0}$, we have $\bm{L}_{\widetilde{G}^{\mu}} \cdot \bm{1}^{\tilde{n}}=(\{\bm{P}(\mu)\}^{\ell})^T \: \bm{L}_G\: \{\bm{P}(\mu)\}^{\ell}\cdot \bm{1}^{\tilde{n}} = (\{\bm{P}(\mu)\}^{\ell})^T \: \bm{L}_G\: c\cdot\bm{1}^n=\bm{0}$. This means $\bm{1}^{\tilde{n}}\in \operatorname{Null}(\bm{L}_{\widetilde{G}^{\mu}})$. However, the corresponding expanded graph $\widetilde{G}^{\mu}$ contains negative weighted edges as we showcase with an example later this section. This leads to further modifications to $\widetilde{G}^{\mu}$ and the construction of PEGP and MSP. 
	
	We include the following proof to show that, by our construction, solving~\eqref{eqn:originalsys} and~\eqref{eqn:expandedsys} are equivalent in the following sense. 
	\begin{theorem}
		Let $\vec{\tilde{x}}$ be a solution to the expanded graph Laplaican system~\eqref{eqn:expandedsys}, then 
		\begin{equation*}
			\vec{x} = \{\bm{P}(\mu)\}^{\ell}\:\vec{\tilde{x}}
		\end{equation*}
		is a solution to the original graph Laplacian system~\eqref{eqn:originalsys}.
	\end{theorem}
	
	\begin{proof}
		Note that,
		\begin{align*}
			\bm{L}_G\:\vec{x} &= \vec{b},\\
			(\{\bm{P}(\mu)\}^{\ell})^T \:\bm{L}_G\:\vec{x} &=(\{\bm{P}(\mu)\}^{\ell})^T\:\vec{b},\\
			(\{\bm{P}(\mu)\}^{\ell})^T \:\bm{L}_G\:\{\bm{P}(\mu)\}^{\ell}\:\vec{\tilde{x}} &=(\{\bm{P}(\mu)\}^{\ell})^T\:\vec{b},\\
			\:\bm{L}_{\widetilde{G}^{\mu}}\:\vec{\tilde{x}} &=\vec{\tilde{b}}, \qquad \text{by \eqref{def:L_ell}}
		\end{align*}
		where $\vec{\tilde{b}} = (\{\bm{P}(\mu)\}^{\ell})^T\:\vec{b}$ and $\vec{x} = \{\bm{P}(\mu)\}^{\ell}\:\vec{\tilde{x}}$, which completes the proof.
	\end{proof}
	
	Solving the original linear system in~\eqref{eqn:originalsys} and the expanded one~\eqref{eqn:expandedsys} are mathematically equivalent. The main motivation of introducing the expanded graph $\widetilde{G}^{\mu}$ is that the expanded graph has a multilayer structure which naturally leads to a smaller diameter compared with the original graph. This property is favored when finding spanning subgraph preconditioners such as the low stretch spanning trees or sparsifiers~\cite{alon1995graph,elkin2008lower,koutis2014approaching}.
	In the MG literature, it is also well-known that the expanded linear system is much better conditioned than the original linear system~\cite{griebel1994multilevel}.

	Due to the construction in~\eqref{def:L_ell}, $\bm{L}_{\widetilde{G}^{\mu}}$ naturally has block structure that corresponds to the multilevel structure of the expanded graph $\widetilde{G}^{\mu}$. The diagonal blocks are the graph Laplacians of each layer, respectively. The off-diagonal blocks of $\bm{L}_{\widetilde{G}^{\mu}}$ represent the edges across the layers, which include negatively weighted edges. The negative edges are unfavorable for common algorithms that find spectral equivalent trees or sparsifiers. In addition, while each row sum of $\bm{L}_{\widetilde{G}^{\mu}}$ is still 0 (similarly, the column sum), $\bm{L}_{\widetilde{G}^{\mu}}$ is not an M-matrix anymore due to the negative weights. To help visualize the negative edges, we present a two-level example to demonstrate the structure of the expanded graph and further analyze the negative edges.
	
	The original graph $G$ is an unweighted (or all the edge weights are $1$) regular 2D lattice with 8 nodes as shown in Figure~\ref{fig:prototype}. Assuming that the $4$ aggregations on the first layer are $\{1,2\},\{3,4\},\{5,6\},\{7,8\}$, we can expand the graph in Figure~\ref{fig:prototype} to have a second layer with four new nodes $9,10,11,12$ as the aggregation representatives in Figure~\ref{fig:prototype_ex}. 
	\begin{figure}[h!]
		\begin{center}
			\includegraphics[width=10cm]{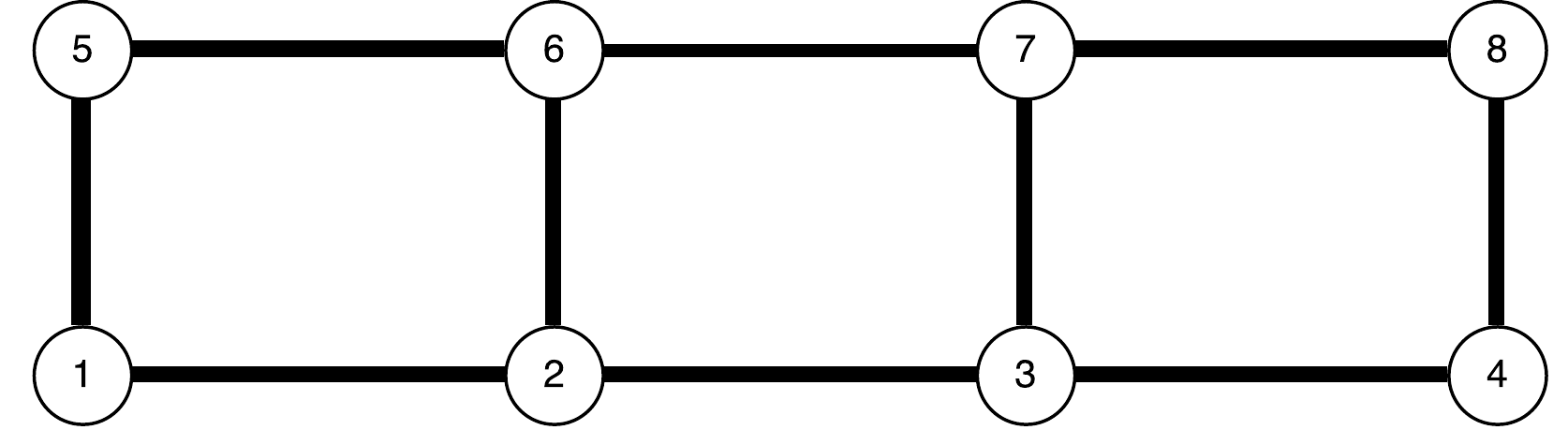}
		\end{center}
		\caption{Original graph G}\label{fig:prototype}
	\end{figure}
		%
	\begin{figure}[h!]
		\begin{center}
			\includegraphics[width=10cm]{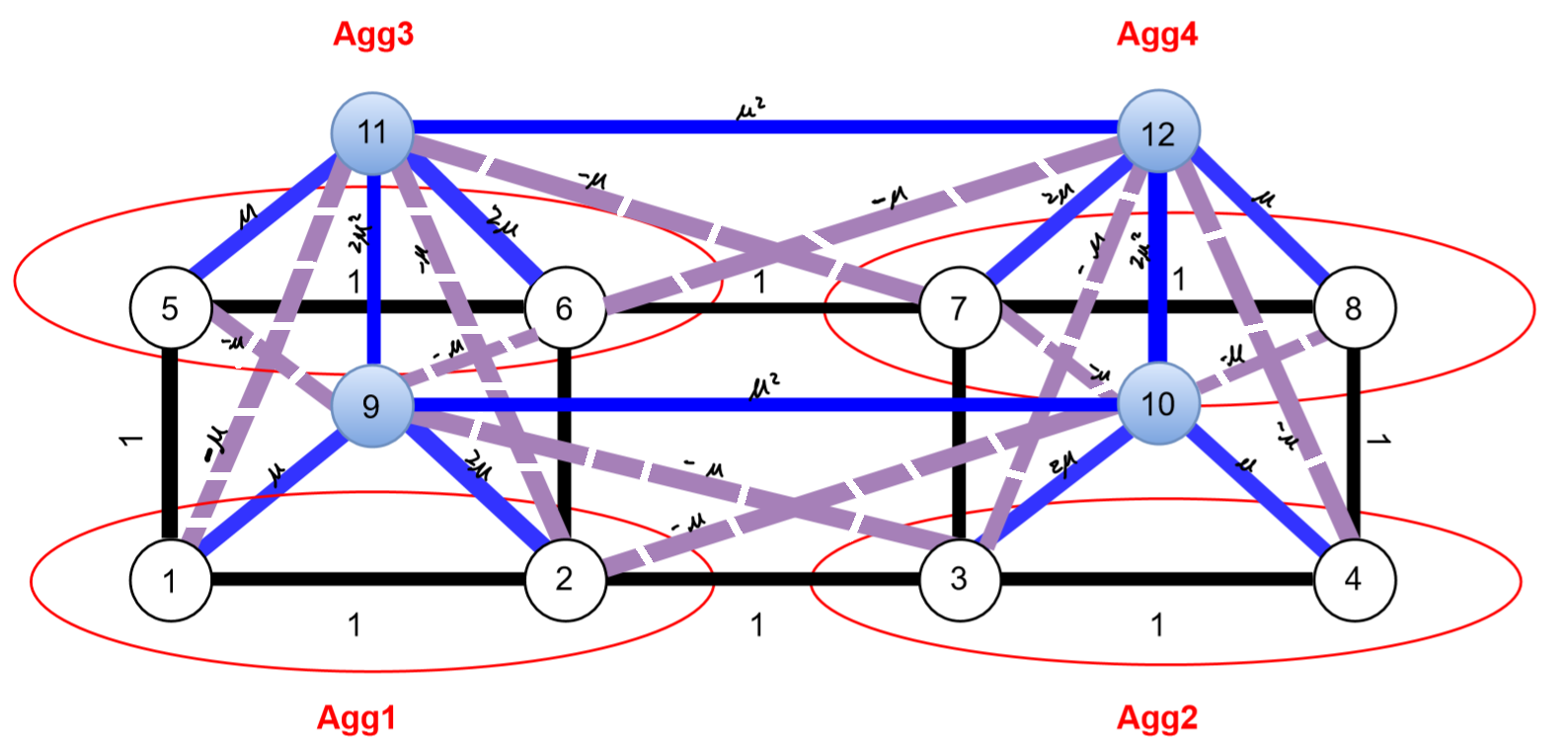}
		\end{center}
		\caption{Two-level expanded graph $\widetilde{G}^{\mu}$}\label{fig:prototype_ex}
	\end{figure}
	
	The prolongation operator according to the first layer aggregations $\{1,2\},\{3,4\},\{5,6\},\{7,8\}$ is,
	\begin{equation}\label{def:P_12}
		\bm{P}_1^2=
		\begin{bmatrix}
			1   & 	0 	& 	0 	& 0	 \\
			1   &   0   & 	0 	& 0  \\
			0   &  1  	&   0   & 0  \\
			0   &  1  	&   0   & 0  \\
			0   &  0	&   1	& 0  \\
			0   &  0    &   1   & 0  \\
			0   &  0    &   0   & 1  \\
			0   &  0    &   0   & 1  
		\end{bmatrix},
	\end{equation}
	and the graph Laplacian of the expanded graph has the following $2 \times 2$ block structure since we have two layers in this example.
	\begin{equation}\label{def:LG_ex}
		\bm{L}_{\widetilde{G}^{\mu}} = 
		\begin{bmatrix}
			\bm{I}\\
			(-\mu\bm{P^2_1})^T 
		\end{bmatrix}\: \bm{L}_G\: 
		\begin{bmatrix}
			\bm{I} & (-\mu\bm{P^2_1}) 
		\end{bmatrix} = \begin{bmatrix}
			\bm{L}_G & \bm{L}_G(-\mu\bm{P^2_1}) \\
			(-\mu\bm{P^2_1})^T\bm{L}_G & \mu^2(\bm{P^2_1})^T\bm{L}_G\bm{P^2_1}
		\end{bmatrix}.
	\end{equation}
	Specifically:
	\[
	\bm{L}_{\widetilde{G}^{\mu}}= 
	\left(\begin{array}{cccccccc|cccc}
		2 & -1 & 0 & 0 & -1& 0 & 0 & 0 & -\mu & 0 & \mu & 0\\
		-1 & 3 & -1 & 0 & 0 & -1 & 0 & 0 & -2\mu & \mu & \mu & 0\\
		0 & -1 & 3 & -1 & 0 & 0 & -1 & 0 & \mu & -2\mu & 0 & \mu\\
		0 & 0 & -1 & 2 & 0 & 0 & 0 & -1 & 0 & -\mu & 0 & \mu\\
		-1 & 0 & 0 & 0 & 2 & -1 & 0 & 0 & \mu & 0 & -\mu & 0\\
		0 & -1 & 0 & 0 & -1 & 3 & -1 & 0 & \mu & 0 & -2\mu & \mu\\
		0 & 0  & -1 & 0 & 0 & -1 & 3 & -1 & 0 & \mu & \mu & -2\mu\\
		0 & 0 & 0 & -1 & 0 & 0 & -1 & 2 & 0 & \mu & 0 & -\mu \\
		\hline
		-\mu & -2\mu & \mu  & 0 & \mu & \mu & 0 & 0 & 3\mu^2 & -\mu^2 & -2\mu^2 & 0\\
		0 & \mu & -2\mu & -\mu & 0 & 0 & \mu & \mu & -\mu^2 & 3\mu^2 & 0 & -2\mu^2 \\	
		\mu & \mu & 0 & 0 & -\mu & -2\mu & \mu & 0 & -2\mu^2 & 0 & 3\mu^2 & -\mu^2 \\
		0 & 0 & \mu & \mu & 0 & \mu & -2\mu & -\mu & 0 & -2\mu^2 & -\mu^2 & 3\mu^2
	\end{array}\right).
	\]
	We can immediately see that there are positive off-diagonal entries in the above matrix which correspond to the negative-weighted edges in the expanded graph.  These negative edges are the edges (dashed purple edges in Figure~\ref{fig:prototype_ex}) connecting vertices on the fine grid (first layer in Figure~\ref{fig:prototype_ex}) and vertices in different aggregations on the coarse grid (second layer in Figure~\ref{fig:prototype_ex}). The magnitude of a negatively weighted edge $\tilde{\omega}_{(u_i,v_j)}(\mu)$, where $u_i$ denotes a node on the first layer in aggregation $i$ and $v_j$ is a node on the second layer represents aggregation $j$, $i\neq j$, is the scaled weight sum of edges that connect vertex $u_i$ with vertices that belong to the aggregate $j$ on the first level. For instance, $\tilde{\omega}_{(1,11)}(\mu)$ in Figure~\ref{fig:prototype_ex} is calculated by summing up edges between node 1 and nodes in the third aggregation (i.e. nodes 5 and 6). Edge $(1,5)$ satisfies the requirement, so $\tilde{\omega}_{(1,11)}(\mu) = \tilde{\omega}_{(1,5)}(\mu)\cdot \mu=1\cdot \mu = \mu.$
	
	
	Although the expanded graph has a small diameter which is favored by existing sparsifiers and spanning subgraph preconditioners, those negatively weighted edges are troublesome since most existing fast algorithms for finding sparsifiers and/or spanning subgraph preconditioners are designed for positive weighted graphs only~\cite{alon1995graph,bern2001support,abraham2012using}.  Therefore, they cannot be applied directly to our expanded graph.  Next, we will argue that we can safely drop those negatively weighted edges and only keep the positively weighted edges on the expanded graph.  The scaling parameter $\mu$ plays an essential role here to guarantee the results positively weighted graph is spectrally equivalent to the expanded graph. 
	
	\subsection{Extracting the Positive Subgraph}\label{sec:H_mu}
	As mentioned in the previous section, with the negative edges in the expanded graph, we cannot directly apply the existing algorithms for constructing spectrally equivalent sparsifiers or subgraph preconditioners on it. We propose to extract only the positive edges from the expanded graph to construct $\widetilde{G}^{\mu}$'s positively weighted subgraph $\widetilde{H}^{\mu}$. Therefore, any algorithms that find good quality sparsifiers, such as in~\cite{elkin2008lower,silva2015sparse} can be performed on the positive subgraph $\widetilde{H}^{\mu}$. Moreover, we show that $\widetilde{H}^{\mu}$ is spectrally equivalent to the expanded graph $\widetilde{G}^{\mu}$ when the parameter $\mu$ is chosen properly.

	We denote $\widetilde{H}^{\mu}$'s graph laplacian as $\bm{L}_{\widetilde{H}^{\mu}}$ and it is now diagonally dominant. As an example, Figure~\ref{fig:spanningsubgraph} is the subgraph of $\widetilde{G}^{\mu}$ in Figure~\ref{fig:prototype_ex} with only positively weighted edges.  
	\begin{figure}[h!]
		\begin{center}
			\includegraphics[width=10cm]{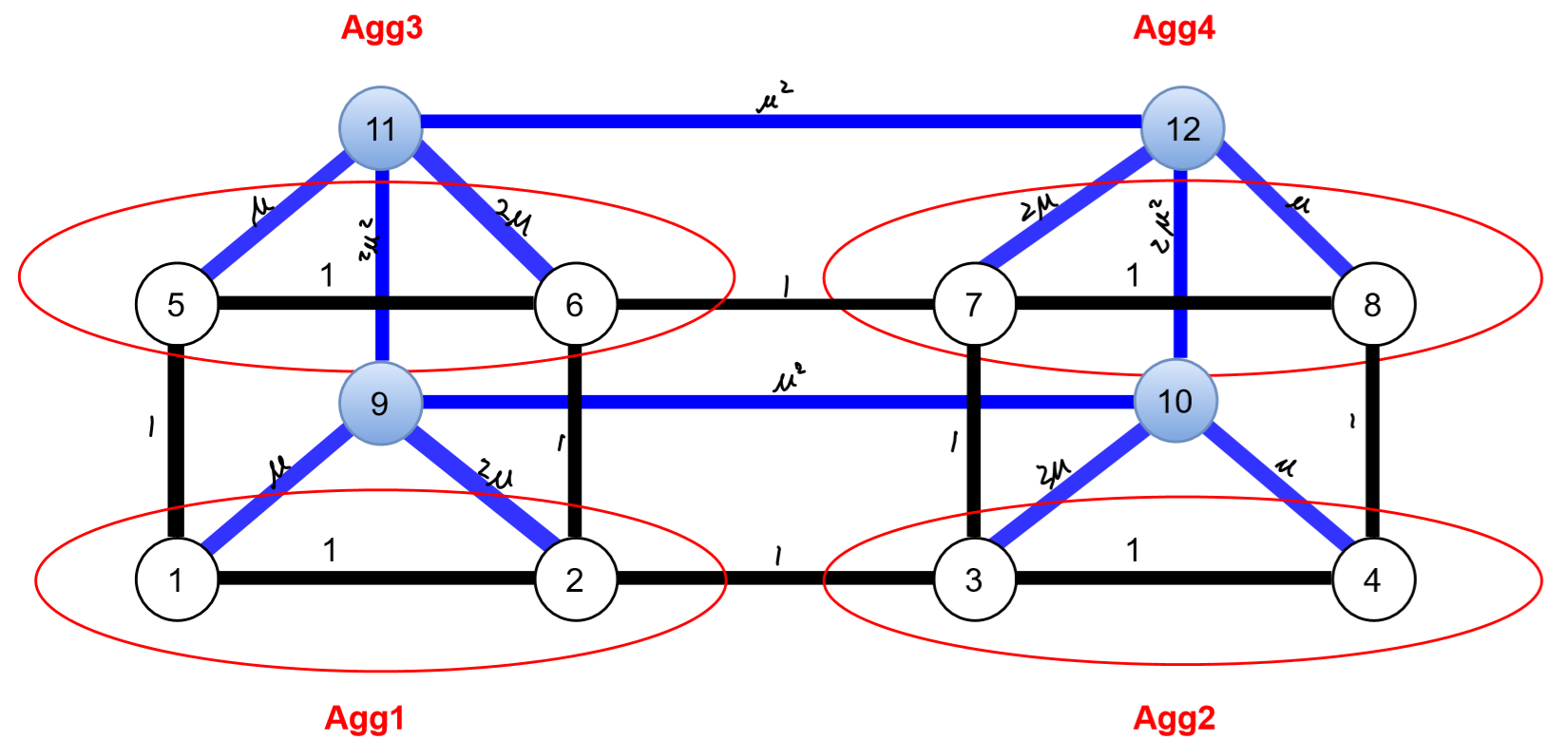}
		\end{center}
		\caption{Spanning positive subgraph $\widetilde{H}^{\mu}$ of $\widetilde{G}^{\mu}$}\label{fig:spanningsubgraph}
	\end{figure}
	For this example, $\bm{L}_{\widetilde{H}^{\mu}}$ is as follows:
	{\tiny 
		\[
		\bm{L}_{\widetilde{H}^{\mu}} = 
		\left(\begin{array}{cccccccc|cccc}
			2+\mu & -1 & 0 & 0 & -1& 0 & 0 & 0 & -\mu & 0 & 0 & 0\\
			-1 & 3+2\mu & -1 & 0 & 0 & -1 & 0 & 0 & -2\mu & 0 & 0 & 0\\
			0 & -1 & 3+2\mu & -1 & 0 & 0 & -1 & 0 & 0 & -2\mu & 0 & 0\\
			0 & 0 & -1 & 2+\mu & 0 & 0 & 0 & -1 & 0 & -\mu & 0 & 0\\
			-1 & 0 & 0 & 0 & 2+\mu & -1 & 0 & 0 & 0 & 0 & -\mu & 0\\
			0 & -1 & 0 & 0 & -1 & 3+2\mu & -1 & 0 & 0 & 0 & -2\mu & 0\\
			0 & 0  & -1 & 0 & 0 & -1 & 3+2\mu & -1 & 0 & 0 & 0 & -2\mu\\
			0 & 0 & 0 & -1 & 0 & 0 & -1 & 2+\mu & 0 & 0 & 0 & -\mu \\
			\hline
			-\mu & -2\mu & 0  & 0 & 0 & 0 & 0 & 0 & 3\mu^2+3\mu & -\mu^2 & -2\mu^2 & 0\\
			0 & 0 & -2\mu & -\mu & 0 & 0 & 0 & 0 & -\mu^2 & 3\mu^2+3\mu & 0 & -2\mu^2 \\	
			0 & 0 & 0 & 0 & -\mu & -2\mu & 0 & 0 & -2\mu^2 & 0 & 3\mu^2+3\mu & -\mu^2 \\
			0 & 0 & 0 & 0 & 0 & 0 & -2\mu & -\mu & 0 & -2\mu^2 & -\mu^2 & 3\mu^2+3\mu
		\end{array}\right).
		\]
	}
	
	The rest of this section is dedicated to show that, with a properly chosen $\mu$, $\bm{L}_{\widetilde{H}^{\mu}}$ is spectral equivalent to $\bm{L}_{\widetilde{G}^{\mu}}$. To start, let us consider the following generalize eigenvalue problem:
	\begin{equation}\label{eqn:general-eigenvalue}
		\bm{L}_{\widetilde{G}^{\mu}} \vec{v} = \lambda \bm{L}_{\widetilde{H}^{\mu}} \vec{v}.
	\end{equation}
	Depending on where the vector $\vec{v}$ resides in, we have three different cases:
	\begin{enumerate}
		\item When $\vec{v}\in\operatorname{Null}(\bm{L}_{\widetilde{H}^{\mu}})$, $\lambda$ could be any number. Since $\bm{L}_{\widetilde{H}^{\mu}}$ is a graph Laplacian matrix of positively weighted graph $\widetilde{H}^{\mu}$, constant vector is its nullspace. (i.e. $\operatorname{Null}(\bm{L}_{\widetilde{H}^{\mu}})=\operatorname{span}\{ \vec{1} \}$). If $\vec{v}\in\operatorname{span}\{ \vec{1} \}$, it is also in the null space of $\bm{L}_{\widetilde{G}^{\mu}}$ as mentioned in Section~\ref{sec:G_tilde}. In this case, $\bm{L}_{\widetilde{G}^{\mu}} \vec{v} = \bm{L}_{\widetilde{H}^{\mu}} \vec{v} = \bm{0}$. Since we cannot bound $\lambda$, we avoid letting $\vec{v}\in\operatorname{Null}(\bm{L}_{\widetilde{H}^{\mu}})$.
		\item When $\vec{v}\perp\operatorname{Null}(\bm{L}_{\widetilde{H}^{\mu}}) $ and $\vec{v}\in\operatorname{Null}(\bm{L}_{\widetilde{G}^{\mu}})$, $\lambda = 0$. To prove that such $\vec{v}$ exists, we claim the following lemma:
		\begin{lemma}\label{lem:dim_G_tilde}
		    Given that graph $G$ and its expanded graph $\widetilde{G}^{\mu}$ are both simply connected, $\dim(\operatorname{Null}(\bm{L}_{\widetilde{G}^{\mu}}))=\tilde{n}-n+1$. Here, $n$ and $\tilde{n}$ represent the number of nodes in $G$ and $\widetilde{G}^{\mu}$ respectively. 
		\end{lemma}
		We include the following proof for lemma~\ref{lem:dim_G_tilde}:
		\begin{proof}
		    We assume the original graph $G$ and its expanded graph $\tilde{G}^{\mu}$ are simply connected.
		    \begin{align*}
		        \bm{L}_{\widetilde{G}^{\mu}} & = (\{\bm{P}(\mu)\}^{\ell})^T\:\bm{L}_G\:\{\bm{P}(\mu)\}^{\ell} \\
		                                      & = (\{\bm{P}(\mu)\}^{\ell})^T\:\bm{B}^T\:\bm{W_{1/2}^T}\:\bm{W_{1/2}}\:\bm{B}\{\bm{P}(\mu)\}^{\ell}
		    \end{align*}
		    We define $\widetilde{\bm{B}}: = \bm{W_{1/2}}\:\bm{B}$, where $\bm{B}$ and $\bm{W_{1/2}}$ are defined in Section~\ref{sec:graph}. Then, we have $$\bm{L}_{\widetilde{G}^{\mu}} = \left(\{\bm{P}(\mu)\}^{\ell}\:\widetilde{\bm{B}}\right)^T\:\left(\{\bm{P}(\mu)\}^{\ell}\:\widetilde{\bm{B}}\right)$$ and $$rank(\bm{L}_{\widetilde{G}^{\mu}}) = rank\left(\left(\{\bm{P}(\mu)\}^{\ell}\:\widetilde{\bm{B}}\right)^T\:\left(\{\bm{P}(\mu)\}^{\ell}\right)\:\widetilde{\bm{B}}\right) = rank(\{\bm{P}(\mu)\}^{\ell}\:\widetilde{\bm{B}}).$$
		    
		    Notice that $\{\bm{P}(\mu)\}^{\ell}\in\mathbb{R}^{n\times \widetilde{n}}$ is in the echelon form, $rank(\{\bm{P}(\mu)\}^{\ell}) = n$. The rank of weighted incidence matrix $\widetilde{\bm{B}}$ is $n - 1$. Recall  two inequalities: 
		    \begin{align*}
		        rank(\{\bm{P}(\mu)\}^{\ell}\:\widetilde{\bm{B}}) &\leqslant \min\left(rank(\{\bm{P}(\mu)\}^{\ell}\right), rank(\widetilde{\bm{B}})) = \min(n, n-1) = n-1 \\
		        rank(\{\bm{P}(\mu)\}^{\ell}\:\widetilde{\bm{B}}) &\geqslant rank(\{\bm{P}(\mu)\}^{\ell} +  rank(\widetilde{\bm{B}}) - n = n + (n-1) -n = n-1 
		    \end{align*}
		    We can get that $rank(\{\bm{P}(\mu)\}^{\ell}\:\widetilde{\bm{B}}) = rank(\bm{L}_{\widetilde{G}^{\mu}})=n-1$. Therefore, the dimension of $\operatorname{Null}(\bm{L}_{\widetilde{G}^{\mu}}) = \tilde{n} - (n - 1) = \tilde{n}-n+1$.
		\end{proof}
		This goes to show that besides the constant vector that is in both $\operatorname{Null}(\bm{L}_{\widetilde{H}^{\mu}})$ and $\operatorname{Null}(\bm{L}_{\widetilde{G}^{\mu}})$, there are $\tilde{n}-n$ other linearly independent vectors that are in $\operatorname{Null}(\bm{L}_{\widetilde{G}^{\mu}})$. Therefore, there exist vectors $\vec{v}$ in nullspace of $\bm{L}_{\widetilde{G}^{\mu}}$ but not in nullspace of $\bm{L}_{\widetilde{H}^{\mu}}$. When $\vec{v}\perp\operatorname{Null}(\bm{L}_{\widetilde{H}^{\mu}}) $ and $\vec{v}\in\operatorname{Null}(\bm{L}_{\widetilde{G}^{\mu}})$, the left-hand-side is always 0 whereas $\bm{L}_{\widetilde{H}^{\mu}} \vec{v}$ is never 0, thus $\lambda$ has to be 0. Since this is a trivial case, we also avoid letting $\vec{v}\perp\operatorname{Null}(\bm{L}_{\widetilde{H}^{\mu}}) $ and $\vec{v}\in\operatorname{Null}(\bm{L}_{\widetilde{G}^{\mu}})$.
		
		\item When $\vec{v}\perp\operatorname{Null}(\bm{L}_{\widetilde{G}^{\mu}})$, there are finite number of positive $\lambda$'s. Define the condition number of $\bm{L}_{\widetilde{H}^{\mu}}^\dagger\bm{L}_{\widetilde{G}^{\mu}},$ as $\kappa(\bm{L}_{\widetilde{H}^{\mu}}^\dagger\bm{L}_{\widetilde{G}^{\mu}})=\dfrac{\lambda_{\max}(\bm{L}_{\widetilde{H}^{\mu}}^\dagger\bm{L}_{\widetilde{G}^{\mu}})}{\lambda_{\min}(\bm{L}_{\widetilde{H}^{\mu}}^\dagger\bm{L}_{\widetilde{G}^{\mu}})}$. We include the following Lemmas~\ref{lem:upper}~\ref{lem:lower-mu}~\ref{lem:lower} and Theorem~\ref{thm:MSP_cond} to show that $\kappa(\bm{L}_{\widetilde{H}^{\mu}}^\dagger\bm{L}_{\widetilde{G}^{\mu}})$ is bounded with proper $
		\mu$'s.
	\end{enumerate}
	
	
	Now we are ready to discuss the bounds of the generalized eigenvalue problem~\eqref{eqn:general-eigenvalue}. First, we show that eigenvalues of $\bm{L}_{\widetilde{H}^{\mu}}^\dagger\bm{L}_{\widetilde{G}^{\mu}}$ are bounded by $1$ from above. 
	\begin{lemma}\label{lem:upper}
		For any $\mu > 0$ and $\vec{v}\in \mathbb{R}^{\widetilde{n}}$, we have $(\bm{L}_{\widetilde{G}^{\mu}} \vec{v}, \vec{v}) \leqslant (\bm{L}_{\widetilde{H}^{\mu}}\vec{v}, \vec{v})$. This implies that the upper bound of $\lambda(\bm{L}_{\widetilde{H}^{\mu}}^\dagger\bm{L}_{\widetilde{G}^{\mu}})$ is 1, when $ \vec{v} \perp \operatorname{Null}(\bm{L}_{\widetilde{H}^{\mu}})$.
	\end{lemma}
	\begin{proof}
		
		We analyze the upper bound by writing out the graph Laplacian quadratic form of $\bm{L}_{\widetilde{G}^{\mu}}\in \mathbb{R}^{\widetilde{n}\times \widetilde{n}}$. For any $\vec{v} \in \mathbb{R}^{\widetilde{n}}$,
		\begin{align*}
			(\bm{L}_{\widetilde{G}^{\mu}} \vec{v}, \vec{v})& = \sum\limits_{(i,j)\in E(\widetilde{G}^{\mu})} \tilde{\omega}_{(i,j)}(\mu)(\vec{v}_i-\vec{v}_j)^2\\
			& = \sum\limits_{(i,j)\in E^+(\widetilde{G}^{\mu})} \tilde{\omega}_{(i,j)}(\mu)(\vec{v}_i-\vec{v}_j)^2 - \sum\limits_{(i,j)\in E^-(\widetilde{G}^{\mu})} \tilde{\omega}_{(i,j)}(\mu)(\vec{v}_i-\vec{v}_j)^2\\
			&=(\bm{L}_{\widetilde{H}^{\mu}}\vec{v}, \vec{v}) - (\bm{L}_{\mu}^-\vec{v}, \vec{v})\\
			&\leqslant (\bm{L}_{\widetilde{H}^{\mu}}\vec{v}, \vec{v}).
		\end{align*}
		Here, $E^+(\widetilde{G}^{\mu})$ and $E^-(\widetilde{G}^{\mu})$ are defined in Section~\ref{sec:Pre}. $\bm{L}_{\mu}^-$ denotes the graph Laplacian of graph induced by edges in $E^-(\widetilde{G}^{\mu})$. $\tilde{\omega}_{(i,j)}(\mu)$ represents the magnitude of the edge $(i,j)$ that depends on $\mu$, as defined in Section~\ref{sec:G_tilde}. 
		
		When we restrict $ \vec{v} \perp \operatorname{Null}(\bm{L}_{\widetilde{H}^{\mu}})$, we can divide $(\bm{L}_{\widetilde{H}^{\mu}}\vec{v}, \vec{v})$ on both sides of the inequality and conclude that the upper bound of $\lambda(\bm{L}_{\widetilde{H}^{\mu}}^\dagger\bm{L}_{\widetilde{G}^{\mu}})$ is 1. This completes the proof.
	\end{proof}
	
	\begin{remark}
		Notice that we attain the upper bound 1 when $(\bm{L}_{\widetilde{G}^{\mu}} \vec{v}, \vec{v}) = (\bm{L}_{\widetilde{H}^{\mu}}\vec{v}, \vec{v})$. This corresponds to the case when $ \vec{v} \in \operatorname{Null}(\bm{L}_{\mu}^-)$. It is easy to see that the dimension of $\operatorname{Null}(\bm{L}_{\mu}^-)$ equals to the number of connected components in graph of $\bm{L}_{\mu}^-$. 
	\end{remark}

	To prove for the lower bound of $\lambda(\bm{L}_{\widetilde{H}^{\mu}}^\dagger\bm{L}_{\widetilde{G}^{\mu}})$, it is important that we choose scalar $\mu$ properly. The following lemma establishes the relationship between $\bm{L}^-_{\mu}$ and $\bm{L}_{\widetilde{H}^{\mu}}$, which is essential for the lower bound.
	\begin{lemma}\label{lem:lower-mu}
		For any scalar $0<\rho\leqslant1$, there exists $0<\mu^*(\rho)<1$ such that for any $0<\mu<\mu^*(\rho)$,
		$$(\bm{L}^-_{\mu} \vec{v}, \vec{v})\leqslant\rho \, (\bm{L}_{\widetilde{H}^{\mu}}\vec{v}, \vec{v}), \quad  \forall \ \vec{v} \in \mathbb{R}^{\widetilde{n}} \ \text{and} \ \vec{v}\perp \operatorname{Null}(\bm{L}_{\widetilde{H}^{\mu}}).
		$$
	\end{lemma}
	\begin{proof}
		Define $c_{\vec{v}}(\mu):=\dfrac{(\bm{L}^-_{\mu} \vec{v}, \vec{v})}{(\bm{L}_{\widetilde{H}^{\mu}}\vec{v}, \vec{v})}=\dfrac{\sum\limits_{(i,j)\in E^-(\widetilde{G}^{\mu})} \tilde{\omega}_{(i,j)}(\mu)(\vec{v}_i-\vec{v}_j)^2 }{\sum\limits_{(i,j)\in E^+(\widetilde{G}^{\mu})} \tilde{\omega}_{(i,j)}(\mu)(\vec{v}_i-\vec{v}_j)^2}$. $c_{\vec{v}}(\mu)$ is well-defined since $\vec{v} \perp \operatorname{Null}(\bm{L}_{\widetilde{H}^{\mu}}) = \operatorname{span}\{ \vec{1} \}$. In addition, since $c_{\vec{v}}(\mu)$ is a ratio, without loss of generality, we only need to consider  $\|\vec{v}\|=\vec{1}$. 
		
		Notice that if $ \vec{v} \in \operatorname{Null}(\bm{L}_{\widetilde{G}^{\mu}})$, $(\bm{L^-_{\mu}} \vec{v}, \vec{v}) = (\bm{L}_{\widetilde{H}^{\mu}}\vec{v}, \vec{v})$ which proves the result with $\rho = 1$.
		
		According to the construction of graph $\widetilde{G}^{\mu}$, $(i, j)\in E^-(\widetilde{G}^{\mu})$ only consists of negative edges that across different layers. However, $(i, j)\in E^+(\widetilde{G}^{\mu})$ includes positive edges connecting nodes both on the same level and across different levels. Due to our construction, 
		we can observe that, 
		$$ c_{\vec{v}}(0) = 0 \text{ and } c_{\vec{v}}(1) \leqslant 1.
		$$
		When $\mu=0$, $E^+(\widetilde{G}^{\mu})$ includes edges in the original graph $G$ which the weights are not multiplied by $\mu$, therefore the denominator of $c_{\vec{v}}(\mu)$ is nonzero. On the other hand, all the edges in $E^-(\widetilde{G}^{\mu})$ are scaled by $\mu$, which makes the numerator zero. Hence, $c_{\vec{v}}(0) = 0$.
		
		When $\mu=1$, both positive and negative edges that are across the layers are scaled by the same power of $\mu$, but the denominator includes more positive edges that are on the same level, therefore we have $c_{\vec{v}}(1) \leqslant 1$.
		
		We can also see that $c_{\vec{v}}(\mu)$ is continuous on $[0, 1]$ with respect to $\mu$ for any fixed $\vec{v}\perp\vec{1}$ and $\|\vec{v}\| = \vec{1}$.
		For such a fixed $\vec{v}$, use the fact that $c_{\vec{v}}(\mu)$ is right continuous at 0, for any given $0<\rho\leqslant1$, we can arrive at the conclusion that there exists a $\mu_{\vec{v}}^*(\rho)>0$, such that
		$$ c_{\vec{v}}(\mu)\leqslant\rho \text{ for any } 0<\mu<\mu_{\vec{v}}^*(\rho)
		<1.	$$
		Since the set $\{ \vec{v} \in \mathbb{R}^{\widetilde{n}} \, | \, \vec{v}\perp\vec{1} ,\: \|\vec{v}\| = \vec{1} \}$ is compact, this implies
		$$c_{\vec{v}}(\mu)\leqslant\rho \text{ for any } \vec{v}\perp\vec{1} ,\: \|\vec{v}\| = \vec{1} \text{ and }0 < \mu<\mu^*(\rho),
		$$
		where $\mu^*(\rho):=\min_{\{ \vec{v} \in \mathbb{R}^{\widetilde{n}} \, | \, \vec{v}\perp\vec{1} ,\: \|\vec{v}\| = \vec{1} \}}\mu_{\vec{v}}^*(\rho) > 0$. 
		Therefore, for any $\mu$ such that $0 < \mu<\mu^*(\rho)$, we have
		$$ (\bm{L}^-_{\mu} \vec{v}, \vec{v})\leqslant\rho \, (\bm{L}_{\widetilde{H}^{\mu}}\vec{v}, \vec{v}), \quad \forall\: \vec{v}\perp\vec{1},\: \|\vec{v}\| = \vec{1}\text{ and } 0<\rho\leqslant1.
		$$
		This completes the proof.
	\end{proof}
	From lemma~\ref{lem:lower-mu}, the lower bound of $\lambda(\bm{L}_{\widetilde{H}^{\mu}}^\dagger\bm{L}_{\widetilde{G}^{\mu}})$ follows naturally,
	\begin{lemma}\label{lem:lower}
		When $\vec{v}\perp \operatorname{Null}(\bm{L}_{\widetilde{H}^{\mu}})$, the lower bound of $\lambda(\bm{L}_{\widetilde{H}^{\mu}}^\dagger\bm{L}_{\widetilde{G}^{\mu}})$ is $1-\rho$, for any $0<\mu<\mu^*(\rho)$, where $\mu^*(\rho)$ is defined in Lemma~\ref{lem:lower-mu}.
	\end{lemma}
	\begin{proof}
		As proved in Lemma~\ref{lem:lower-mu}, $(\bm{L}^-_{\mu} \vec{v}, \vec{v})\leqslant\rho \ (\bm{L}_{\widetilde{H}^{\mu}}\vec{v}, \vec{v}), \forall\: \vec{v}\perp\vec{1},\: \|\vec{v}\| = \vec{1}\text{ and } 0<\rho\leqslant1.$ Therefore, we have
		\begin{align*}
			(\bm{L}_{\widetilde{G}^{\mu}} \vec{v}, \vec{v})=(\bm{L}_{\widetilde{H}^{\mu}}\vec{v}, \vec{v}) - (\bm{L}_{\mu}^-\vec{v}, \vec{v})
			\geqslant (\bm{L}_{\widetilde{H}^{\mu}}\vec{v}, \vec{v}) - \rho(\bm{L}_{\widetilde{H}^{\mu}}\vec{v}, \vec{v})
			=(1-\rho) (\bm{L}_{\widetilde{H}^{\mu}}\vec{v}, \vec{v}).
		\end{align*}
		which completes the proof. 
	\end{proof}

Together with the result for upper bound, we have the following theorem which shows the spectral equivalence between the expanded graph $\widetilde{G}^{\mu}$ and its positive weighted subgraph $\widetilde{H}^{\mu}$.
\begin{theorem}\label{thm:MSP_cond}
	Given a simply connected graph $G$, let $\widetilde{G}^{\mu}$ be the expanded graph defined in~\eqref{def:L_ell} and~$\widetilde{H}^{\mu}$ be the subgraph of~$\widetilde{G}^{\mu}$ consisting only positive edges. Then for a given $0< \rho \leq 1$, there exists a constant $\mu^*(\rho)$ defined in Lemma~\ref{lem:lower-mu} such that, for any $0<\mu<\mu^*(\rho)$, when $\vec{v}\perp\operatorname{Null}(\bm{L}_{\widetilde{H}^{\mu}})$, we have
	\begin{equation*}
		(1-\rho)  (\bm{L}_{\widetilde{H}^{\mu}}\vec{v}, \vec{v}) \leq (\bm{L}_{\widetilde{G}^{\mu}} \vec{v}, \vec{v}) \leq (\bm{L}_{\widetilde{H}^{\mu}}\vec{v}, \vec{v}).
	\end{equation*}
	Moreover, when $\vec{v}\perp\operatorname{Null}(\bm{L}_{\widetilde{G}^{\mu}})$, then $0 < \rho < 1$ and the corresponding finite condition number is $\kappa(\bm{L}_{\widetilde{H}^{\mu}}^{\dagger}\bm{L}_{\widetilde{G}^{\mu}})\leqslant \frac{1}{1-\rho}$.
\end{theorem}
\begin{proof}
	By the proofs of Lemmas~\ref{lem:upper} and~\ref{lem:lower}, we have $(1-\rho)\cdot (\bm{L}_{\widetilde{H}^{\mu}}\vec{v}, \vec{v}) \leqslant (\bm{L}_{\widetilde{G}^{\mu}}\vec{v}, \vec{v}) \leqslant (\bm{L}_{\widetilde{H}^{\mu}}\vec{v}, \vec{v})$ when $\vec{v}\perp\operatorname{Null}(\bm{L}_{\widetilde{H}^{\mu}})$. If we restrict $\vec{v}\perp\operatorname{Null}(\bm{L}_{\widetilde{G}^{\mu}})$, we have $0<\rho<1$. We can then divide the lower bound and get the condition number  $\kappa(\bm{L}_{\widetilde{H}^{\mu}}^\dagger\bm{L}_{\widetilde{G}^{\mu}}) \leqslant \dfrac{1}{1-\rho}, 0<\rho<1.$ 
\end{proof}
\begin{remark}
	According to Thoerem~\ref{thm:MSP_cond}, in practice, we can always choose a proper $\mu$ to guarantee a small condition number $1/(1-\rho)$. However, we want to point out that our results is an existence result, i.e., for any given $0 < \rho < 1$, we can choose small enough $\mu$ in order to guarantee the spectral equivalence between the graphs $\widetilde{G}^{\mu}$ and $\widetilde{H}^{\mu}$.  How to explicitly compute $\mu$ for a given $\rho$ is still open and a subject of our ongoing work.
\end{remark}

To get a more in-depth understanding of the selections of scaling parameter $\mu$ in practice, we demonstrate it on different graphs with various sizes. In all cases, we use MWM as the aggregation method (as suggested in Section~\ref{sec:Pre}). We construct the expanded graph to the maximum level possible (with only two nodes on the top level) and record $\kappa((\bm{L}_{\widetilde{H}^{\mu}}^\dagger\bm{L}_{\widetilde{G}^{\mu}})$ (the ratio between the largest  eigenvalue $\lambda_{\max}(\bm{L}_{\widetilde{H}^{\mu}}^\dagger\bm{L}_{\widetilde{G}^{\mu}})$ and the smallest nonzero eigenvalue $\lambda_{\min}(\bm{L}_{\widetilde{H}^{\mu}}^\dagger\bm{L}_{\widetilde{G}^{\mu}})$) 
with different graph sizes $n$ and proper scaling factors $\mu$.

We calculate the condition numbers on two types of structured graphs, 2D regular grids and ring graphs. Examples of a 2D regular grid with $16$ nodes and a ring graph with $10$ nodes and average degree of 4 are shown in Figure~\ref{fig:2D-grids} and Figure~\ref{fig:ring1} respectively. 
\begin{figure}[h!]
	\begin{center}
		\includegraphics[width=5cm]{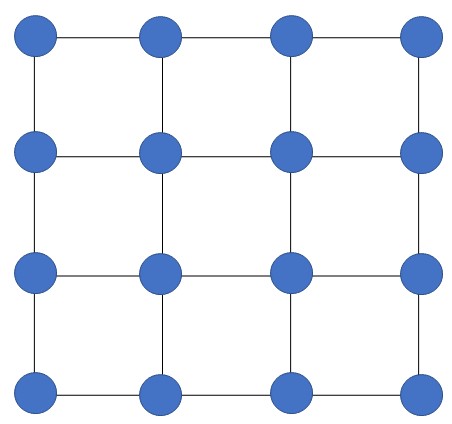}
	\end{center}
	\caption{A 2D regular grid with $|V| = 16$}\label{fig:2D-grids}
\end{figure}

\begin{figure}[h!]
	\begin{center}
		\includegraphics[width=5cm]{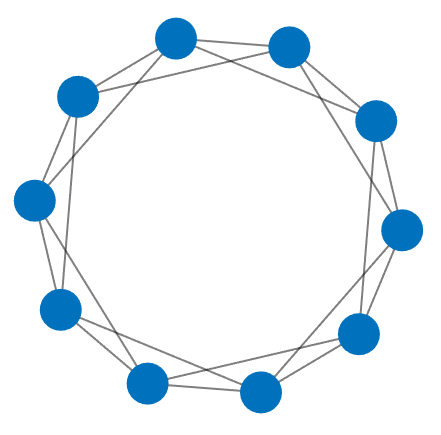}
	\end{center}
	\caption{A ring graph with $|V| = 10$, and average degree $=4$}\label{fig:ring1}
\end{figure}

We first analyze the condition numbers when picking a relatively large $\mu=0.8$. In this case, $\mu$ is a constant that does not vary with the graph size $n$:
\begin{table}[h]
	\begin{center}
		{
			\caption{Analysis of condition number on 2D regular grids and ring graphs with average degree $=4$, $\mu=0.8$}\label{tab:mu0.8}
			\begin{tabular}{||c|c|c||c|c||}
				\hline\hline
				\multicolumn{3}{||c||}{Graph Basics}& \multicolumn{2}{||c||}{$\kappa(\bm{L}_{\widetilde{H}^{\mu}}^\dagger\bm{L}_{\widetilde{G}^{\mu}})$}\\ \hline
				Size of $\bm{L_1}\in \mathbb{R}^{n\times n}$ & Level & Size of $\bm{L}_{\widetilde{G}^{\mu}}\in \mathbb{R}^{\tilde{n}\times \tilde{n}}$ & 2D Grids & Ring Graphs \\ \hline
				16*16    & 4    & 30*30    & 4.5713 &     3.5001       \\ \hline
				64*64     & 6   & 126*126      &14.4171 &     16.9643           \\ \hline
				256*256     & 8   & 510*510      & 61.8819 &  83.5291       \\ \hline
				1024*1024  & 10      & 2046*2046       &204.5351 & 348.3507            \\ \hline \hline
			\end{tabular}
		}
	\end{center}
\end{table}

				
From Table~\ref{tab:mu0.8}, we can observe that with large $\mu$, the condition numbers for both types of graphs grow with the graph size and are not bounded. This corresponds to our conclusion in Theorem~\ref{thm:MSP_cond}. Therefore, we experiment with smaller $\mu$'s in the following tests:

\begin{table}[h!]
	\begin{center}
		{
			\caption{Analysis of condition number on 2D regular grids and ring graphs with average degree $=4$, $\mu=\dfrac{1}{\sqrt{n}}$}\label{tab:mu}
			\begin{tabular}{||c|c|c||c|c||}
				\hline\hline
				\multicolumn{3}{||c||}{Graph Basics}& \multicolumn{2}{||c||}{$\kappa(\bm{L}_{\widetilde{H}^{\mu}}^\dagger\bm{L}_{\widetilde{G}^{\mu}})$}\\ \hline
				Size of $\bm{L_1}\in \mathbb{R}^{n\times n}$ & Level & Size of $\bm{L}_{\widetilde{G}^{\mu}}\in \mathbb{R}^{\tilde{n}\times \tilde{n}}$ & 2D Grids & Ring Graphs \\ \hline
				16*16    & 4    & 30*30    & 1.5680 &     2.6021       \\ \hline
				64*64     & 6   & 126*126      &1.4128 &     3.1655          \\ \hline
				256*256     & 8   & 510*510      & 1.1959 &  2.1026        \\ \hline
				1024*1024  & 10      & 2046*2046       &1.0926 & 1.6033            \\ \hline \hline
			\end{tabular}
		}
	\end{center}
\end{table}

				
We can observe from Table~\ref{tab:mu} that setting $\mu=\dfrac{1}{\sqrt{n}}$ on the regular grids and ring graphs gives bounded condition number $\kappa(\bm{L}_{\widetilde{H}^{\mu}}^\dagger\bm{L}_{\widetilde{G}^{\mu}})$ even with increasing graph size. It verifies Theorem~\ref{thm:MSP_cond} numerically that there exists a small enough $\mu$ such that $\kappa(\bm{L}_{\widetilde{H}^{\mu}}^\dagger\bm{L}_{\widetilde{G}^{\mu}})$ is bounded from above. Moreover, with the chosen $\mu$, it can be observed that as graph sizes grow larger, the condition numbers decrease slightly for both graph types.

\subsection{Finding Spectral Sparsifiers of $\widetilde{H}^{\mu}$}\label{sec:MSP}
Now that PEGP $\widetilde{H}^{\mu}$ has small diameter from its multilevel structure and positive weights only, it makes a good candidate as a preconditioner for the expanded graph $\widetilde{G}^{\mu}$. Moreover, it is easy to apply any off-the-shelf algorithms to find spanning trees and sparsifiers on PEGP, which we leads us to finding spectral sparsifiers that further improve the performance. To explain the construction, we first review the idea of stretch and its relation to condition number. In~\cite{alon1995graph}, the stretch of an edge $(i,j)$ in $\widetilde{H}^{\mu}$ is defined as:
\begin{equation}\label{def:stretch-edge}
\text{stretch}_{\widetilde{T}^\mu}(i,j)=\omega_{(i,j)}\left(\sum_{(u,v)\in p(i, j)}1/\omega_{(u,v)}\right),
\end{equation}
where $\omega_{(i,j)}$ is the edge weight of edge $(i,j)$ and $p(i,j)$ is the path that connects nodes $i$ and $j$ in the tree $\widetilde{T}^\mu$. The stretch of the entire positive subgraph $\widetilde{H}^{\mu}$ is then:
\begin{equation}\label{def:stretch-graph}
\text{stretch}_{\widetilde{T}^\mu}(\widetilde{H}^{\mu})=\sum_{(i,j)\in E^+ (\widetilde{G}^{\mu})}\text{stretch}_{\widetilde{T}^\mu}(i,j).
\end{equation}
It is proved in Section 2.2 of~\cite{phdthesis} that $\kappa(\bm{L}_{\widetilde{T}^\mu}^{\dagger}\bm{L}_{\widetilde{H}^{\mu}})=\|\bm{K}_{\widetilde{T}^\mu}\|_2^2\leqslant\|\bm{K}_{\widetilde{T}^\mu}\|_F^2=\text{stretch}_{\widetilde{T}^\mu}(\widetilde{H}^{\mu})$. Here, $\bm{K}_{\widetilde{T}^\mu}=\bm{\widetilde{B}}_{\widetilde{T}^\mu}^\dagger\cdot\bm{\widetilde{B}}_{\widetilde{G}^\mu}$, and $\bm{\widetilde{B}}_{\widetilde{T}^\mu}$ and $\bm{\widetilde{B}}_{\widetilde{G}^\mu}$ are the weighted incidence matrices of tree $\widetilde{T}^\mu$ and expanded graph $\widetilde{G}^\mu$ respectively.  Combine our finding in Theorem~\ref{thm:MSP_cond} with the result in~\cite{phdthesis}, we can get $(1 -\rho)(\bm{L}_{\widetilde{T}^\mu}\: \vec{v}, \vec{v}) \leqslant (1 -\rho)(\bm{L}_{\widetilde{H}^{\mu}}\: \vec{v}, \vec{v})\leqslant (\bm{L}_{\widetilde{G}^{\mu}}\: \vec{v}, \vec{v})\leqslant (\bm{L}_{\widetilde{H}^{\mu}}\: \vec{v}, \vec{v}) \leqslant \text{stretch}_{\widetilde{T}^\mu}(\widetilde{H}^{\mu})(\bm{L}_{\widetilde{T}^\mu}\: \vec{v}, \vec{v})$. Therefore, the condition number $\kappa(\bm{L}_{\widetilde{T}^\mu}^{\dagger}\bm{L}_{\widetilde{G}^{\mu}}) = \dfrac{\text{stretch}_{\widetilde{T}^\mu}(\widetilde{H}^{\mu})}{1-\rho}$. 

In addition to using spanning tree, researchers have found that sparsifier created by adding back spectrally critical off-tree edges to low-stretch spanning tree offers a practically efficient, nearly-linear time preconditioner for large-scale, real-world
graph problems~\cite{abraham2012using,feng2016spectral}.

Although there are existing spectral sparsifiers that can be applied to the PEGP $\widetilde{H}^{\mu}$ directly, to further demonstrate that the benefit of constructing $\tilde{H}^{\mu}$, we propose a tree-structured multilevel sparsifier preconditioner for $\widetilde{H}^{\mu}$ which can be constructed in a straightforward manner. Thanks to the fact that the positive subgraph $\widetilde{H}^{\mu}$ of the expanded graph $\widetilde{G}^{\mu}$ has small diameter from its multilevel structure and positive weights only, it is easy for us to define a low-stretch spanning sparsifier. Specifically, we extract only the edges on the top coarse level, edges between the levels, and edges inside of the aggregations of the original graphs from the positive subgraph. For example, the positive subgraph in~\ref{fig:spanningsubgraph} can be further extracted to:

\begin{figure}[h!]
	\begin{center}
		\includegraphics[width=10cm]{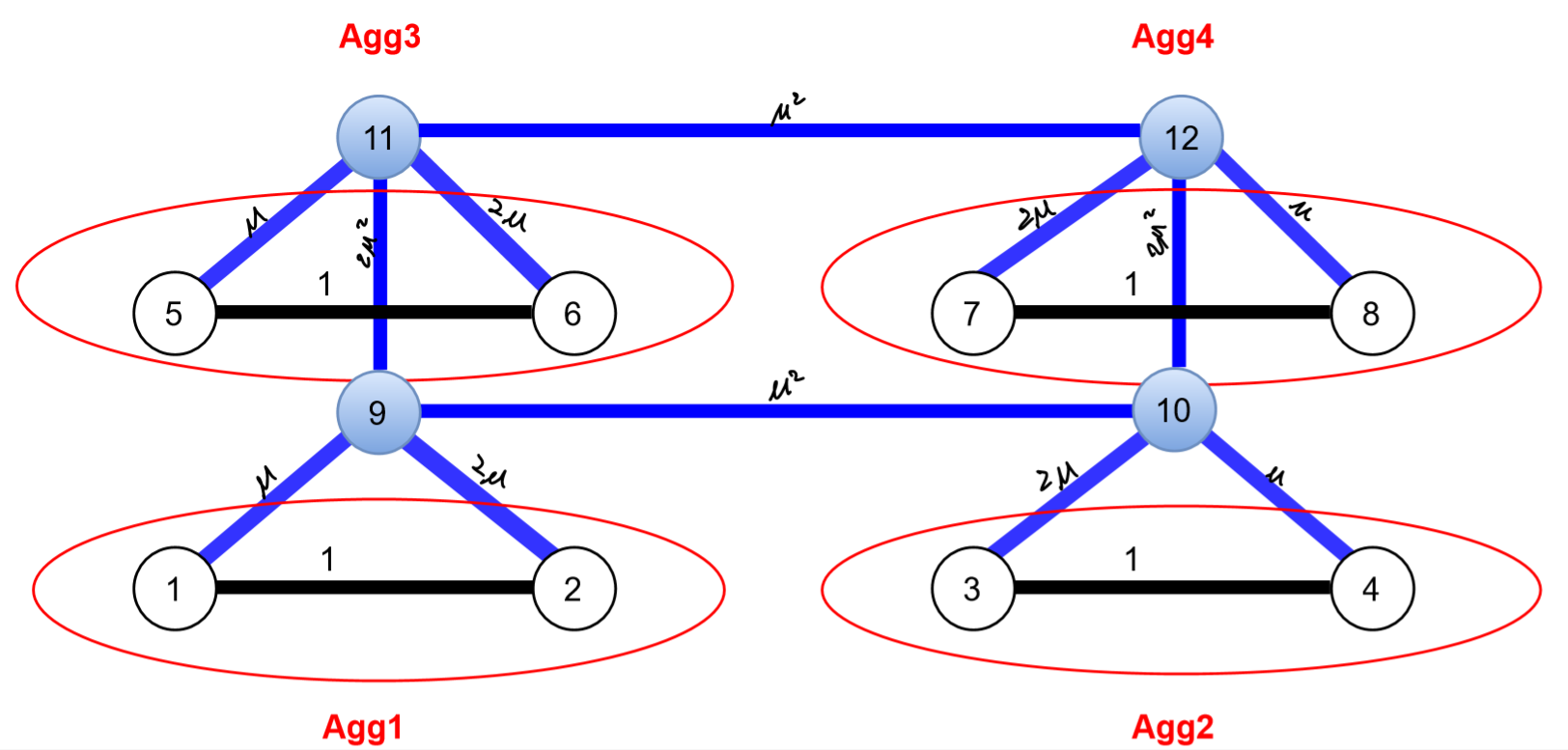}
	\end{center}
	\caption{2-level spanning tree-structured sparsifier MSP of PEGP $\widetilde{H}^{\mu}$ with aggregations defined in Equation~\ref{def:P_12}}\label{fig:LT}
\end{figure}

Here, the MSP we build here is a sparsifier rather than a spanning tree. When computing $\text{stretch}_{\widetilde{T}^\mu}(\widetilde{H}^{\mu})$, $\text{dist}_{\widetilde{T}^\mu}(i,j)$ is not unique. We choose the path with the lightest edge weight sum as $\text{dist}_{\widetilde{T}^\mu}(i,j)$ in this case. Since we use MWM coarsening scheme, most aggregations are matchings of size 2 . MSP's tree-like structure guarantees that the Gaussian elimination procedure introduces no fill-ins to the Schur complement when eliminating each aggregation less than four nodes from the base level.  

From the standpoint of achieving the best theoretical bound for $\kappa(\bm{L}_{\widetilde{T}^\mu}^{\dagger}\bm{L}_{\widetilde{G}^{\mu}})$, we need to pick large $\mu$ to get small $\text{stretch}_{\widetilde{T}^\mu}(\widetilde{H}^{\mu})$ but small-scale $\mu$ to control $\dfrac{1}{1-\rho}$. In practice, we found that $\mu=1/\sqrt{n}$ leads to steady iteration steps and efficient solving time for our test graphs in~\ref{sec:num}.

\section{Numerical Results}\label{sec:num}
In this section, we verify the spectral approximating property of the PEGP $\widetilde{H}^{\mu}$ by using it as a preconditioner for solving the expanded graph $\widetilde{G}^{\mu}$, which is equivalent to the original graph $G$. Moreover, we also demonstrate the performance of MSP to show its practical applications. We conduct numerical experiments on different types of graph Laplacians; some are related to discrete PDEs, and the rest are derived from real-world networks. For all the tests, we solve the expanded graph Laplacian system in~\eqref{eqn:expandedsys}, which is equivalent to the original system~\eqref{eqn:originalsys}. We pair the model problems with a low-frequency right-hand-side $\bm{\tilde{b}}\in \mathbb{R}^{\tilde{n}}=(\{\bm{P}(\mu)\}^{\ell})^T \bm{b}, \text{where } \bm{b}=[1,1,\dots,1-n]^T$. We compare the performance of the proposed preconditioners to the Combinatorial Multigrid (CMG) preconditioner ~\cite{KOUTIS}. CMG borrows the structure and operations of multigrid algorithms, but its setup phase is based on support theory. It has been noted by many authors for its robustness in practice for various real-world problems~\cite{livne2012lean,krishnan2011multigrid,garyfallou2018combinatorial}. Since CMG does not work with graphs with negative edges, we use it to solve the original system~\eqref{eqn:originalsys}.

We choose MWM coarsening scheme to build the multilevel structure described in Section~\ref{sec:agg} and add  the leftover nodes to the aggregate with the heaviest incident edge. Therefore, there are no isolated nodes on any level. For all the experiments below, we use flexible generalized minimum residual method (FGMRES)~\cite{saad1986gmres} as the solver to solve \eqref{eqn:expandedsys} to a relatively small tolerance $1e-8$ for the scalability tests in Sections~\ref{sec:regular-grid} and~\ref{sec:ring-graph} and more practical choice $1e-6$ for real-world graphs in Section~\ref{sec:realworld}.
  
Numerical experiments are conducted using a $1.1$GHz Intel
i7-10710U CPU with 16 GB of RAM.  The software used for testing PEGP and MSP is written by the authors and implemented in
Matlab code. The package used to test for CMG method can be found in~\cite{cmg-solver}, which is a C code integrated in Matlab interface. 
We report the following results in each table.  The
``steps" column reports the number of iterations needed for the
residual to reach certain tolerance, ``level" represents the total number of levels for the MSP, $\mu$ is a scalar of choice from 0 to 1, and we report ``time" as the total
CPU times in seconds.  We set the maximum number of iterations to 1000 and denote any cases where the solver takes more than 1000 iterations as ``--".

\subsection{On Regular Grids}\label{sec:regular-grid}
We first tested the performance of the preconditioners on unweighted graph Laplacians of 2D regular uniform grids in Table~\ref{tab:grids}. This relates to solving a Poisson equation with Neumann boundary condition on a 2D square.


\begin{table}[h]
	\begin{center}
		{\tiny
			\caption{Performance on Regular Grids $\epsilon=10^{-8}$}\label{tab:grids}
			\begin{tabular}{|c||c|c||c|c||c|c||c|c||c|c|c|} \hline \hline
				& \multicolumn{4}{|c||}{level = 2, $\mu=1/\sqrt{n}$} & \multicolumn{4}{|c|}{level = max, $\mu=1/\sqrt{n}$} &\multicolumn{3}{|c|}{}\\ \hline
				& \multicolumn{2}{|c||}{PEGP} &  \multicolumn{2}{|c||}{MSP}&\multicolumn{2}{|c||}{PEGP}& \multicolumn{2}{|c|}{MSP} & \multicolumn{3}{|c|}{CMG}\\ \hline
				$n$ &  steps  & time& steps  & time &  steps  & time & steps  & time & level & time & steps\\ \hline \hline
				$2^8$ & 12 & 0.0094 & 48 & 0.0208 & 7 & 0.0092 &  58 & 0.0133 & 4 & 14  & 0.0393\\
				$2^{10}$ & 10  & 0.0333 & 51   & 0.0618 & 6 & 0.0527 &  93 &  0.0793& 5 & 15   & 0.0877 \\
				$2^{12}$ & 9  & 0.1377 & 51  & 0.2748 & 5 & 0.1933	&  151  & 0.2281 & 6  & 16 & 0.3923\\
				$2^{14}$ &  8  & 0.5092 & 55  & 1.0923 & 4 & 0.7758 &  248   & 1.2778 	 & 7 & 16  & 1.0066 \\
				\hline \hline
			\end{tabular}
		}
	\end{center}
\end{table}

There are two main observations from the numerical results: (1) As shown in Theorem~\ref{thm:MSP_cond}, with a small enough $\mu$, $\kappa(\bm{L}^{\dagger}_{\widetilde{H}^{\mu}}\bm{L}_{\widetilde{G}^{\mu}})$ can be controlled, and this is confirmed by the stable iteration counts even when the amount of level is fixed to 2. When extending to the maximal level, the PEGP still has stable iteration counts, contributing to the short CPU time needed for computation. (2) It is easy to construct MSP from PEGP due to its small diameter. MSP has low stretch and tree structure, which makes each solving step faster to run. When level = 2, the CPU time scales linearly with the problem size. When level = max, MSP's computation time scales in $O(n\log n)$.  In addition, the CPU time of our proposed preconditioners are faster or comparable with CMG's performance for this model testing problem.

\subsection{On Random Watts-Strogatz Graphs} \label{sec:ring-graph}
For the second example, we use Watts Strogatz~\cite{watts1998collective} model and set the rewiring probability $ \beta = 1/\sqrt{n}$ and the mean node degree to be $4$ to produce ring-like graphs as in~\ref{fig:ring}.  The condition numbers of the graph Laplacians of the ring graphs also grow rapidly when the size of the graphs increases.

\begin{figure}[h!]
	\begin{center}
		\includegraphics[width=0.6\textwidth]{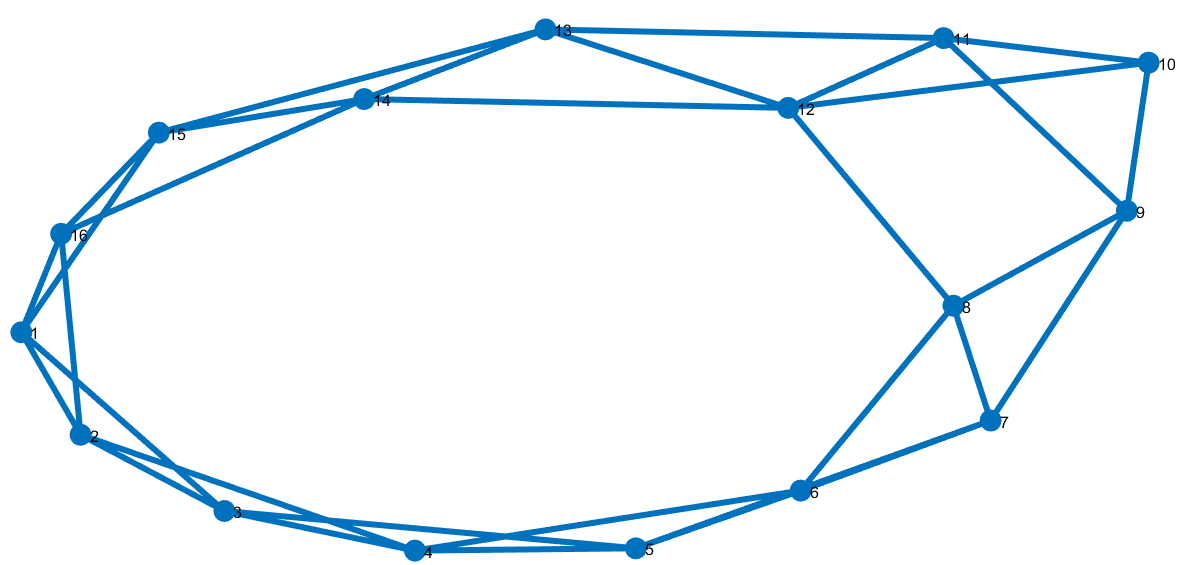}
		\caption{Example of a Watts-Strogatz graph with $n=16$, mean node degree$=4$, and rewiring probability $ \beta = 1/4$}\label{fig:ring}
	\end{center}
\end{figure}

\begin{table}[h]
	\begin{center}
		{\tiny
			\caption{Performance on Solving Watts-Strogatz Graphs $\epsilon=10^{-8}$, rewiring probability$ = 1/\sqrt{n}$}\label{tab:rings}
			\begin{tabular}{|c||c|c||c|c||c|c||c|c||c|c|c|} \hline \hline
				& \multicolumn{4}{|c||}{level = 2, $\mu=1/\sqrt{n}$}&\multicolumn{4}{|c||}{level = max, $\mu=1/\sqrt{n}$}& \multicolumn{3}{|c|}{}\\ \hline
				& \multicolumn{2}{|c||}{PEGP} &  \multicolumn{2}{|c||}{MSP}&\multicolumn{2}{|c||}{PEGP}& \multicolumn{2}{|c||}{MSP} &  \multicolumn{3}{|c|}{CMG}\\ \hline
				$n$ &  steps  & time& steps  & time &  steps  & time & steps  & time & level &  steps  & time\\ \hline \hline
				$2^8$    & 11  & 0.0076 & 42  & 0.0180   &  6  & 0.0065 &  89 & 0.0073 &4 & 22  & 0.0124\\
				$2^{10}$ & 9  & 0.0197 & 51  &  0.0557  & 5   & 0.0253 &  171 &  0.0335 & 5 & 30   & 0.0243\\
				$2^{12}$ & 9  & 0.0619 & 57  & 0.1814   & 5   & 0.0960 & 344 & 0.1196 &  6  & 46 & 0.1515\\
				$2^{14}$ & 7  & 0.2402 & 62  & 0.6645   & 4   & 0.3097 & 704 & 0.9943 & 7 & 37  & 2.2038	\\
				\hline \hline
			\end{tabular}
		}
	\end{center}
\end{table}


Similar observations can be made from this set of tests. In addition to the analysis in Table~\ref{tab:mu}, the stable (even slightly decreasing) iteration count for $\widetilde{H}^{\mu}$ numerically confirms Theorem~\ref{thm:MSP_cond} that PEGP provides a spectral equivalent positively weighted expanded graph when $\mu$ is small. MSP again has an increasing iteration count. However, each step takes a much faster time to solve than PEGP and CMG. In fact, in all cases, the CPU time needed for PEGP and MSP outperform the CMG preconditioner. As the ring graph gets larger, the CMG method's CPU time and convergence steps tend to increase more drastically than the proposed preconditioners. 

\subsection{On Real-world Graphs}\label{sec:realworld}
In this section, we perform tests on real-world networks from the Stanford large network datasets collection~\cite{snapnets} and the University of Florida sparse matrix collection~\cite{davis2011university}.

We pre-processed the graphs as follows. The largest connected component of each graph is extracted, any self-loops from the extracted component are discarded, and edge weights of the component are modified to be their absolute values to satisfy the requirements of the CMG algorithm. We also made the largest connected components undirected if they were originally directed. In Tables~\ref{tab:snapnets} and~\ref{tab:UFnets}, the basic information of pre-processed graphs collected from the Stanford large network datasets collection and University of Florida sparse matrix collection are presented.

\begin{table}[htp]
	\caption{Largest connected components of the networks from Stanford large network datasets collection}\label{tab:snapnets}
	\begin{center}
		{\small
			\begin{tabular}{|c||c|c|l|} 
				\hline \hline
				&   nodes   &   edges &   Description  \\ \hline \hline
				ca-GrQc         & 4,158   &  13,422 &  Collaboration network of Arxiv General Relativity\\
				ego-Facebook         & 4,039    &  88,234 &  Social circles from Facebook (anonymized)\\
				feather-lastfm-social     &  7,624   &  27,806 &  Social network of LastFM users from Asia \\ 		
				p2p-Gnutella04     & 10,876	 & 39,994  & Gnutella peer to peer network from August 4 2002 \\
				\hline \hline
			\end{tabular}
		}
	\end{center}
\end{table}%

\begin{table}[htp]
	\caption{Largest connected components of the networks from University of Florida sparse matrix collection}\label{tab:UFnets}
	\begin{center}
		{\footnotesize
			\begin{tabular}{|c||c|c|l|} 
				\hline \hline
				&   nodes   &   edges &   Description  \\ \hline \hline
				airfoil\_2d               & 14,214       & 259,688         & Unstructured 2D mesh (airfoil)\\
				Pres\_Poisson         & 14,822     & 715,804         & A computational fluid dynamics problem from ACUSIM, Inc\\
				Oberwolfach/gyro\_k      & 17,361       & 1,021,159         & Oberwolfach model reduction benchmark collection\\
				tsyl201                  & 20,685       & 2,454,957         & Matrix Representation of TSYL201, part of condeep cylinder\\
				\hline \hline
			\end{tabular}
		}
	\end{center}
\end{table}%

\begin{table}[h]
	\begin{center}
		{\small
			\caption{Performance on Solving Real World Graphs $\epsilon=10^{-6}$, level = max} \label{tab:real-results}
			\begin{tabular}{|c|c||c|c||c|c||c|c|c|} \hline \hline
				&  & \multicolumn{2}{|c||}{PEGP} &  \multicolumn{2}{|c||}{MSP} &\multicolumn{3}{|c|}{CMG} \\ \hline
				Networks & $\mu$ & steps  & time & steps  & time & level & steps & time  \\ \hline \hline
				ca-GrQc & 0.8 & 26 & 1.8934 & 61  &   0.5268  & 6 & 32  &  0.1154 \\
				ego-Facebook & 0.8 & 36 & 0.7868  & 57  &   0.5017 & 9 & 38 & 0.2531  \\
				feather-lastfm-social & 0.8 & 28 & 1.8435  & 65  &   0.9154  & 5 & 20   & 1.1955 \\
				p2p-Gnutella04 & 0.8 & 15 & 4.6773 & 64  &   2.2723  &  7  & 13 & 3.5261 \\
				airfoil\_2d & 0.8 & 117 & 72.7376 & 148 & 5.0942 & 6 & 7 & 5.9355 \\
				Pres\_Poisson & 0.8 & 64 & 77.7364 & 105 & 3.7668 & 7 & 8 & 7.3866\\
				Oberwolfach/gyro\_k & 0.8 & 42 & 32.7343 & 32 & 4.0394 & 6 & 14 & 11.8209\\
				tsyl201 & 0.8 & 44 & 187.0472 & 54 & 3.2941& 6 & 28 & 4.2295\\
				\hline \hline
			\end{tabular}
		}
	\end{center}
\end{table}


Notice that there is a trade-off in picking $\mu$'s. As proved in Sections~\ref{sec:H_mu} and \ref{sec:MSP}, a small enough $\mu$ keeps $\kappa(\bm{L}_{\widetilde{H}^{\mu}}^{\dagger}\bm{L}_{\widetilde{G}^{\mu}})$ bounded but simultaneously increase $\text{stretch}_{\widetilde{T}^\mu}(\widetilde{H}^{\mu})$, hence the condition number $\kappa(\bm{L}_{\widetilde{T}^\mu}^{\dagger}\bm{L}_{\widetilde{G}^{\mu}})$. Since the structures and properties of the real-life graphs vary much more than the regular grids and the ring graphs, PEGP sometimes has many high-degree nodes, making the matrix inversion time-consuming at each step for larger and denser graphs. Although it converges with fewer steps, the total solving time might not be optimal in practice. MSP is a subgraph of PEGP and its tree-like structure introduces no fill-ins during Gaussian elimination. Therefore each step takes a shorter time than PEGP. As shown in Table~\ref{tab:real-results}, in practice, we choose large $\mu$ (for example, 0.8) to control the stretch and boost the performance of MSP. Especially for larger and denser networks, while PEGP might still be difficult and expensive to solve due to the graph structure, MSP with properly chosen $\mu$ demonstrate a more robust performance than CMG in terms of the CPU time. 

\section{Conclusion and Future Work}\label{Sec:conclusion}
While the spectral sparsifiers preconditioners have theoretical guarantees from the support theory for solving graph Laplacians, they are challenging to implement due to the randomness in finding low-stretch spanning trees/sparsifiers. On the other hand, multilevel methods are deterministic and can achieve nearly linear optimal complexity, but they require PDE-based assumptions and lack theoretical justifications for general graphs. In this paper, we proposed a new method to construct preconditioners that combines the advantages of both. 

We adopt the idea of MG and expand the original graph $G$ to a multilevel structured graph $\widetilde{G}^{\mu}$ with a small diameter. However, we introduce negative edges in $\widetilde{G}^{\mu}$, which make it challenging to apply existing algorithms to find spanning trees/sparsifiers. In Section~\ref{sec:H_mu}, based on the support theory for spanning subgraphs, we define a PEGP $\widetilde{H}^{\mu}$ of $\widetilde{G}^{\mu}$, such that $\bm{L}_{\widetilde{H}^{\mu}}$ is spectrally equivalent to $\bm{L}_{\widetilde{G}^{\mu}}$. Now that
$\widetilde{H}^{\mu}$ has positive weights, small diameter, and multilevel structure, making it easy to construct spanning trees/sparsifiers on. To demonstrate this, we propose the MSP in Section~\ref{sec:MSP} for solving graph Laplacians robustly in practice.


In Section~\ref{sec:num}, we test PEGP and MSP as preconditioners for solving graph Laplacians on 2D regular grids, random Watts-Strogatz graphs. We observe that the new preconditioners take stabler iteration steps and less CPU time than the CMG preconditioner. The solving time scales linearly with the problem size with the proposed preconditioners. Compared to the CMG preconditioner, the proposed approaches are more robust and effective on real-life graphs from various sources. 

As for future directions, we would like to extend our work in the following ways. The parameter $\mu$ plays an important in both theory and practice. Smaller $\mu$'s in $\widetilde{H}^{\mu}$ guarantee a faster convergence rate at each step, while larger $\mu$'s, on the other hand, reduce the stretch and minimize the total computation time for MSP. Therefore, we plan to investigate how to properly chose $\mu$ in practical applications. We also want to test the proposed preconditioners on a broader range of real-life graphs and summarize the graph properties that lead to better performance using the new preconditioners. Finally, we will compare the proposed methods to other spanning trees/sparsifiers besides CMG.

\section*{Acknowledgments}
The work of X. Hu is partially supported by the National Science Foundation under grant DMS-1812503 and CCF-1934553

\bibliographystyle{siam}
\bibliography{MSP}

\begin{thebibliography}{10}

\bibitem{abraham2008nearly}
{\sc Ittai Abraham, Yair Bartal, and Ofer Neiman}, {\em Nearly tight low
  stretch spanning trees}, in 2008 49th Annual IEEE Symposium on Foundations of
  Computer Science, IEEE, 2008, pp.~781--790.

\bibitem{abraham2012using}
{\sc Ittai Abraham and Ofer Neiman}, {\em Using petal-decompositions to build a
  low stretch spanning tree}, in Proceedings of the forty-fourth annual ACM
  symposium on Theory of computing, 2012, pp.~395--406.

\bibitem{agarwal2006higher}
{\sc Sameer Agarwal, Kristin Branson, and Serge Belongie}, {\em Higher order
  learning with graphs}, in Proceedings of the 23rd international conference on
  Machine learning, ACM, 2006, pp.~17--24.

\bibitem{alon1995graph}
{\sc Noga Alon, Richard~M Karp, David Peleg, and Douglas West}, {\em A
  graph-theoretic game and its application to the k-server problem}, SIAM
  Journal on Computing, 24 (1995), pp.~78--100.

\bibitem{bern2001support}
{\sc Marshall Bern, John~R Gilbert, Bruce Hendrickson, and Nhat Nguyen}, {\em
  Support-graph preconditioners}, in SIAM Journal on Matrix Analysis and
  Applications, Citeseer, 2001.

\bibitem{boman2001spanning}
{\sc Erik Boman and Bruce Hendrickson}, {\em On spanning tree preconditioners},
  Manuscript, Sandia National Lab, 3 (2001).

\bibitem{boman2003support}
{\sc Erik~G Boman and Bruce Hendrickson}, {\em Support theory for
  preconditioning}, SIAM Journal on Matrix Analysis and Applications, 25
  (2003), pp.~694--717.

\bibitem{AMG_1984}
{\sc A.~Brandt, S.~F. McCormick, and J.~W. Ruge}, {\em Algebraic multigrid
  ({AMG}) for sparse matrix equations}, in Sparsity and Its Applications, D.~J.
  Evans, ed., Cambridge University Press, Cambridge, 1984.

\bibitem{Brannick.J;Chen.Y;Kraus.J;Zikatanov.L.2013a}
{\sc James Brannick, Yao Chen, Johannes Kraus, and Ludmil Zikatanov}, {\em
  Algebraic multilevel preconditioners for the graph laplacian based on
  matching in graphs}, SIAM Journal on Numerical Analysis, 51 (2013),
  pp.~1805--1827.

\bibitem{Bronstein.M;Bruna.J;LeCun.Y;Szlam.A;Vandergheynst.P2017a}
{\sc Michael~M Bronstein, Joan Bruna, Yann LeCun, Arthur Szlam, and Pierre
  Vandergheynst}, {\em Geometric deep learning: going beyond euclidean data},
  IEEE Signal Processing Magazine, 34 (2017), pp.~18--42.

\bibitem{Bruna.J;Zaremba.W;Szlam.A;LeCun.Y2013a}
{\sc Joan Bruna, Wojciech Zaremba, Arthur Szlam, and Yann LeCun}, {\em Spectral
  networks and locally connected networks on graphs}, arXiv preprint
  arXiv:1312.6203,  (2013).

\bibitem{cao2014new}
{\sc Mengfei Cao, Christopher~M Pietras, Xian Feng, Kathryn~J Doroschak, Thomas
  Schaffner, Jisoo Park, Hao Zhang, Lenore~J Cowen, and Benjamin~J Hescott},
  {\em New directions for diffusion-based network prediction of protein
  function: incorporating pathways with confidence}, Bioinformatics, 30 (2014),
  pp.~i219--i227.

\bibitem{DSD:2013}
{\sc Mengfei Cao, Hao Zhang, Jisoo Park, Noah~M. Daniels, Mark~E. Crovella,
  Lenore~J. Cowen, and Benjamin Hescott}, {\em Going the distance for protein
  function prediction: A new distance metric for protein interaction networks},
  PLoS ONE, 8 (2013), pp.~1--12.

\bibitem{HodgeRank}
{\sc C.~Colley, J.~Lin, X.~Hu, and S.~Aeron}, {\em Algebraic multigrid for
  least squares problems on graphs with applications to hodgerank}, in 2017
  IEEE International Parallel and Distributed Processing Symposium Workshops
  (IPDPSW), May 2017, pp.~627--636.

\bibitem{cowen2021random}
{\sc Lenore~J Cowen, Xiaozhe Hu, Junyuan Lin, Yue Shen, and Kaiyi Wu}, {\em
  Random-walk based approximate k-nearest neighbors algorithm for diffusion
  state distance}, in International Conference on Large-Scale Scientific
  Computing, Springer, 2021, pp.~3--15.

\bibitem{d2013adaptive}
{\sc Pasqua D'Ambra and Panayot~S Vassilevski}, {\em Adaptive amg with
  coarsening based on compatible weighted matching}, Computing and
  Visualization in Science, 16 (2013), pp.~59--76.

\bibitem{davis2011university}
{\sc Timothy~A Davis and Yifan Hu}, {\em The university of florida sparse
  matrix collection}, ACM Transactions on Mathematical Software (TOMS), 38
  (2011), pp.~1--25.

\bibitem{elkin2008lower}
{\sc Michael Elkin, Yuval Emek, Daniel~A Spielman, and Shang-Hua Teng}, {\em
  Lower-stretch spanning trees}, SIAM Journal on Computing, 38 (2008),
  pp.~608--628.

\bibitem{feng2016spectral}
{\sc Zhuo Feng}, {\em Spectral graph sparsification in nearly-linear time
  leveraging efficient spectral perturbation analysis}, in Proceedings of the
  53rd Annual Design Automation Conference, 2016, pp.~1--6.

\bibitem{galil1986ev}
{\sc Zvi Galil, Silvio Micali, and Harold Gabow}, {\em An o(ev$\backslash$logv)
  algorithm for finding a maximal weighted matching in general graphs}, SIAM
  Journal on Computing, 15 (1986), pp.~120--130.

\bibitem{garyfallou2018combinatorial}
{\sc Dimitrios Garyfallou, Nestor Evmorfopoulos, and Georgios Stamoulis}, {\em
  A combinatorial multigrid preconditioned iterative method for large scale
  circuit simulation on gpu s}, in 2018 15th International Conference on
  Synthesis, Modeling, Analysis and Simulation Methods and Applications to
  Circuit Design (SMACD), IEEE, 2018, pp.~209--212.

\bibitem{gremban1996combinatorial}
{\sc Keith~D Gremban}, {\em Combinatorial preconditioners for sparse,
  symmetric, diagonally dominant linear systems}, PhD thesis, Carnegie Mellon
  University, 1996.

\bibitem{griebel1994multilevel}
{\sc Michael Griebel}, {\em Multilevel algorithms considered as iterative
  methods on semidefinite systems}, SIAM Journal on Scientific Computing, 15
  (1994), pp.~547--565.

\bibitem{he2006tensor}
{\sc Xiaofei He, Deng Cai, and Partha Niyogi}, {\em Tensor subspace analysis},
  in Advances in neural information processing systems, 2006, pp.~499--506.

\bibitem{Henaff.M;Bruna.J;LeCun.Y2015a}
{\sc Mikael Henaff, Joan Bruna, and Yann LeCun}, {\em Deep convolutional
  networks on graph-structured data}, arXiv preprint arXiv:1506.05163,  (2015).

\bibitem{hirani2010least}
{\sc Anil~N Hirani, Kaushik Kalyanaraman, and Seth Watts}, {\em Least squares
  ranking on graphs}, arXiv preprint arXiv:1011.1716,  (2010).

\bibitem{hu2019adaptive}
{\sc Xiaozhe Hu, Junyuan Lin, and Ludmil~T Zikatanov}, {\em An adaptive
  multigrid method based on path cover}, SIAM Journal on Scientific Computing,
  41 (2019), pp.~S220--S241.

\bibitem{GLTD}
{\sc B.~Jiang, C.~Ding, J.~Tang, and B.~Luo}, {\em Image representation and
  learning with graph-laplacian tucker tensor decomposition}, IEEE Transactions
  on Cybernetics,  (2018), pp.~1--10.

\bibitem{jiang2011statistical}
{\sc Xiaoye Jiang, Lek-Heng Lim, Yuan Yao, and Yinyu Ye}, {\em Statistical
  ranking and combinatorial hodge theory}, Mathematical Programming, 127
  (2011), pp.~203--244.

\bibitem{KARYPIS199896}
{\sc George Karypis and Vipin Kumar}, {\em Multilevelk-way partitioning scheme
  for irregular graphs}, Journal of Parallel and Distributed Computing, 48
  (1998), pp.~96--129.

\bibitem{HEM}
{\sc George Karypis and Vipin Kumar}, {\em Kumar, v.: A fast and high quality
  multilevel scheme for partitioning irregular graphs. siam journal on
  scientific computing 20(1), 359-392}, Siam Journal on Scientific Computing,
  20 (1999).

\bibitem{Kelner:2013}
{\sc Jonathan~A. Kelner, Lorenzo Orecchia, Aaron Sidford, and Zeyuan~Allen
  Zhu}, {\em A simple, combinatorial algorithm for solving sdd systems in
  nearly-linear time}, in Proceedings of the Forty-fifth Annual ACM Symposium
  on Theory of Computing, STOC '13, New York, NY, USA, 2013, ACM, pp.~911--920.

\bibitem{Kim.H;Xu.J;Zikatanov.L2003}
{\sc HwanHo Kim, Jinchao Xu, and Ludmil Zikatanov}, {\em A multigrid method
  based on graph matching for convection-diffusion equations}, Numer. Linear
  Algebra Appl., 10 (2003), pp.~181--195.
\newblock Dedicated to the 60th birthday of Raytcho Lazarov.

\bibitem{Kipf.T;Welling.M2016a}
{\sc Thomas~N Kipf and Max Welling}, {\em Semi-supervised classification with
  graph convolutional networks}, arXiv preprint arXiv:1609.02907,  (2016).

\bibitem{cmg-solver}
{\sc Ioannis Koutis}, {\em Ikoutis/cmg-solver: Combinatorial multigrid solver
  for sdd matrices}.

\bibitem{koutis2011nearly}
{\sc Ioannis Koutis, Gary~L Miller, and Richard Peng}, {\em A nearly-m log n
  time solver for sdd linear systems}, in 2011 IEEE 52nd Annual Symposium on
  Foundations of Computer Science, IEEE, 2011, pp.~590--598.

\bibitem{koutis2014approaching}
\leavevmode\vrule height 2pt depth -1.6pt width 23pt, {\em Approaching
  optimality for solving sdd linear systems}, SIAM Journal on Computing, 43
  (2014), pp.~337--354.

\bibitem{KOUTIS}
{\sc Ioannis Koutis, Gary~L. Miller, and David Tolliver}, {\em Combinatorial
  preconditioners and multilevel solvers for problems in computer vision and
  image processing}, Computer Vision and Image Understanding, 115 (2011),
  pp.~1638--1646.
\newblock Special issue on Optimization for Vision, Graphics and Medical
  Imaging: Theory and Applications.

\bibitem{krishnan2011multigrid}
{\sc Dilip Krishnan and Richard Szeliski}, {\em Multigrid and multilevel
  preconditioners for computational photography}, ACM Transactions on Graphics
  (TOG), 30 (2011), pp.~1--10.

\bibitem{steerableGL}
{\sc Boris Landa and Yoel Shkolnisky}, {\em The steerable graph laplacian and
  its application to filtering image data-sets}, CoRR, abs/1802.01894 (2018).

\bibitem{snapnets}
{\sc Jure Leskovec and Andrej Krevl}, {\em {SNAP Datasets}: {Stanford} large
  network dataset collection}.
\newblock \url{http://snap.stanford.edu/data}, June 2014.

\bibitem{li2018adaptive}
{\sc Ruoyu Li, Sheng Wang, Feiyun Zhu, and Junzhou Huang}, {\em Adaptive graph
  convolutional neural networks}, in Proceedings of the AAAI Conference on
  Artificial Intelligence, vol.~32, 2018.

\bibitem{phdthesis}
{\sc Junyuan Lin}, {\em Preconditioned Iterative Methods for Linear Equations
  of Graph Laplacians: Algorithms and Applications}, PhD thesis, Tufts
  University, 2019.

\bibitem{DSD2018}
{\sc Junyuan Lin, Lenore Cowen, Benjamin Hescott, and Xiaozhe Hu}, {\em
  Computing the diffusion state distance on graphs via algebraic multigrid and
  random projections}, Numerical Linear Algebra with Applications, 25 (2018),
  p.~e2156.
\newblock e2156 nla.2156.

\bibitem{livne2012lean}
{\sc Oren~E Livne and Achi Brandt}, {\em Lean algebraic multigrid (lamg): Fast
  graph laplacian linear solver}, SIAM Journal on Scientific Computing, 34
  (2012), pp.~B499--B522.

\bibitem{narita2012tensor}
{\sc Atsuhiro Narita, Kohei Hayashi, Ryota Tomioka, and Hisashi Kashima}, {\em
  Tensor factorization using auxiliary information}, Data Mining and Knowledge
  Discovery, 25 (2012), pp.~298--324.

\bibitem{notay2010aggregation}
{\sc Yvan Notay}, {\em An aggregation-based algebraic multigrid method},
  Electronic transactions on numerical analysis, 37 (2010), pp.~123--146.

\bibitem{ruge1984efficient}
{\sc John Ruge and Klaus St{\"u}ben}, {\em Efficient solution of finite
  difference and finite element equations by algebraic multigrid AMG},
  Gesellschaft f. Mathematik u. Datenverarbeitung, 1984.

\bibitem{ruge1983algebraic}
{\sc John~W Ruge}, {\em Algebraic multigrid {(AMG)} for geodetic survey
  problems}, in Prelimary Proc. Internat. Multigrid Conference, Fort Collins,
  CO, 1983.

\bibitem{Ruge.J;Stuben.K1987}
{\sc J.~W. Ruge and K.~St{{\"u}}ben}, {\em Algebraic multigrid}, in Multigrid
  methods, vol.~3 of Frontiers Appl. Math., SIAM, Philadelphia, PA, 1987,
  pp.~73--130.

\bibitem{saad1986gmres}
{\sc Youcef Saad and Martin~H Schultz}, {\em Gmres: A generalized minimal
  residual algorithm for solving nonsymmetric linear systems}, SIAM Journal on
  scientific and statistical computing, 7 (1986), pp.~856--869.

\bibitem{shaham2018spectralnet}
{\sc Uri Shaham, Kelly Stanton, Henry Li, Ronen Basri, Boaz Nadler, and Yuval
  Kluger}, {\em Spectralnet: Spectral clustering using deep neural networks},
  in International Conference on Learning Representations, 2018.

\bibitem{silva2015sparse}
{\sc Marcel K De~Carli Silva, Nicholas~JA Harvey, and Cristiane~M Sato}, {\em
  Sparse sums of positive semidefinite matrices}, ACM Transactions on
  Algorithms (TALG), 12 (2015), pp.~1--17.

\bibitem{spielman2004nearly}
{\sc Daniel~A Spielman and Shang-Hua Teng}, {\em Nearly-linear time algorithms
  for graph partitioning, graph sparsification, and solving linear systems}, in
  Proceedings of the thirty-sixth annual ACM symposium on Theory of computing,
  2004, pp.~81--90.

\bibitem{Vaidya90}
{\sc Pravin~M. Vaidya}, {\em Solving linear equations with symmetric diagonally
  dominant matrices by constructing good preconditioners}, 1990.

\bibitem{Vassilevski.P2008}
{\sc Panayot~S. Vassilevski}, {\em Multilevel block factorization
  preconditioners}, Springer, New York, 2008.
\newblock Matrix-based analysis and algorithms for solving finite element
  equations.

\bibitem{watts1998collective}
{\sc Duncan~J Watts and Steven~H Strogatz}, {\em Collective dynamics of
  `small-world' networks}, Nature, 393 (1998), pp.~440--442.

\bibitem{xu2017algebraic}
{\sc Jinchao Xu and Ludmil Zikatanov}, {\em Algebraic multigrid methods}, Acta
  Numerica, 26 (2017), pp.~591--721.

\bibitem{young1974note}
{\sc H~Peyton Young}, {\em A note on preference aggregation}, Econometrica:
  Journal of the Econometric Society,  (1974), pp.~1129--1131.

\end{thebibliography}
\end{document}